\documentclass[12pt,reqno]{amsart}
\usepackage{amsmath}
\usepackage{amssymb}
\usepackage[latin1]{inputenc}
\usepackage{hyperref}
\usepackage{fullpage}

\usepackage{mathrsfs}
\usepackage{dsfont}
\usepackage{stmaryrd}

\renewcommand{\epsilon}{\varepsilon}
\renewcommand{\phi}{\varphi}

\DeclareMathOperator{\OO}{O}

\newcommand{\RR}{{\mathbb{R}}}

\newcommand{\Span}{\operatorname{{Span}}}

\newcommand{\rank}{\operatorname{rank}}

\DeclareMathOperator{\E}{\mathbb{E}}

\newcommand{\tra}[1]{{#1}^{\mathrm{t}}}

\newcommand{\dd}{{\mathrm{d}}}            

\newcommand{\Id}{{\mathrm{Id}}}

\renewcommand{\Re}{{\mathfrak{Re}}}
\renewcommand{\Im}{{\mathfrak{Im}}}

\newtheorem{thm}{Theorem}[section]

\newtheorem*{lem*}{Lemma}
\newtheorem{prop}[thm]{Proposition}
\theoremstyle{definition}
\newtheorem{defn}[thm]{Definition}
\theoremstyle{remark}
\newtheorem*{rem}{Remark}
\newtheorem*{thm*}{Theorem}

\numberwithin{equation}{section}

\date{}

\begin{document}
\title[A unitary extension of virtual permutations]
{A unitary extension of virtual permutations}

\author{P. Bourgade, J. Najnudel, A. Nikeghbali
}

\date\today

\begin{abstract}
Analogously to the space of virtual permutations
\cite{Kerov}, we define projective limits of isometries: these
sequences of unitary operators are natural in the sense that they minimize the rank norm between successive matrices of increasing sizes. The space of virtual isometries we construct this way may be viewed as a natural extension of the space of virtual permutations of Kerov, Olshanski and Vershik  (\cite{Kerov}) as well as the space of virtual isometries of Neretin (\cite{Ner}).
We then derive with purely probabilistic methods  an almost sure convergence for these random matrices under the Haar measure:  for
a coherent Haar measure on virtual isometries,  the smallest normalized eigenangles converge a.s. to a point process
whose correlation function is given by the sine kernel.
This almost sure convergence actually holds for a larger class
of measures as is proved by Borodin and Olshanski (\cite{BO}). We give a different proof, probabilist
 in the sense that it makes use of martingale arguments and shows how the eigenangles interlace when going from dimension $n$ to $n+1$.
Our method also proves that for some universal constant $\epsilon> 0$, the rate of convergence is almost surely dominated by $n^{-\epsilon}$ when the dimension $n$ goes to infinity.
 \end{abstract}

\maketitle

\section{Introduction}

In \cite{Kerov}, the virtual permutations have been introduced in order to study some
asymptotic properties of the symmetric group $\mathcal{S}_n$ of order $n \in \mathbb{N}$, when $n$ goes to infinity.
The space $\mathcal{S}^{\infty}$ of virtual permutation can be defined as follows. For $n \geq m \geq 1$, $\pi_{n,m}$ denotes the application
from $\mathcal{S}_n$ to $\mathcal{S}_m$ such that for $\sigma \in \mathcal{S}_n$, $\pi_{n,m}(\sigma)
\in \mathcal{S}_m$ is obtained from $\sigma$ by deleting all the elements of $\{m+1,\dots,n\}$ from its cycle structure;
then $\mathcal{S}^{\infty}$ is the projective limit of $(\mathcal{S}_n)_{n \geq 1}$, i.e. the set of sequences $(\sigma_n)_{n \geq 1}$
of permutations such that for $n \geq 1$, $\sigma_n \in \mathcal{S}_n$ and for $n \geq m \geq 1$, $\sigma_m = \pi_{n,m}(\sigma_n)$.
A virtual permutation $(\sigma_n)_{n \geq 1}$ can naturally be constructed by the so-called
{\it Chinese restaurant process} (see e.g. \cite{Pitman}), as follows:
\begin{itemize}
\item $\sigma_1$ is the unique permutation in $\mathcal{S}_1$;
\item for $n \geq 1$, $\sigma_{n+1}$ is obtained from $\sigma_n$ either by adding $n+1$ as a fixed point,
or by inserting $n+1$ inside a cycle of $\sigma_n$.
\end{itemize}
If the space $\mathcal{S}^{\infty}$ is endowed with the $\sigma$-algebra generated by the coordinates
$(\sigma_n)_{n \geq 1}$, and if $(\mu_n)_{n \geq 1}$ is a sequence of probability measures, $\mu_n$ defined on 
$\mathcal{S}_n$, such that for all $n \geq 1$, $\mu_n$ is the image of $\mu_{n+1}$ by $\pi_{n+1, n}$, then there exists 
a unique probability measure on  $\mathcal{S}^{\infty}$ such that its $n$-th projection is equal to $\mu_n$ for all $n \geq 1$
(this result can be deduced from the Carath\'eodory theorem on extensions of measures). Among all the probability measures on 
$\mathcal{S}^{\infty}$, those which are invariant by conjugation (i.e. for all $n \geq 1$, their projection on $\mathcal{S}_n$ 
is invariant by conjugation with any element of $\mathcal{S}_n$) are called {\it central measures} and they have been 
studied in detail by Tsilevich (see \cite{Tsi} and \cite{Tsi98}). The main result is the following: 
there exists a family $(\mu^x)_{x \in \Sigma}$ of particular central measures (called {\it ergodic measures}), indexed by 
the set 
$$\Sigma = \{ x = (x_1, x_2, \dots, x_k, \dots), \; x_1 \geq x_2 \geq \dots \geq x_k \geq \dots \geq 0, \; x_1 + x_2 + \dots + x_k + \dots \leq 1\},$$
 such that every central measure $\mu$ can be written as follows: 
\begin{equation}
\mu = \int_{\Sigma} \mu^x d \nu(x), \label{nu}
\end{equation}
where $\nu$ is a probability measure on $\Sigma$. Now, if $x \in \Sigma$ is fixed, and 
if $(\sigma_n)_{n \geq 1}$ is a virtual permutation following the ergodic probability distribution $\mu^x$, then its global cycle structure 
determines a random partition of the set of positive integers, and by the general results of Kingman (see \cite{King75}, 
\cite{King78a}, \cite{King78b}), the following properties hold:
\begin{itemize}
\item The sets of the partition are either  singletons (corresponding to fixed points of permutations)
or have a strictly positive asymptotic density.
\item For all $k \geq 1$, the $k$-th largest cycle length of $\sigma_n$ (zero if $\sigma_n$ has less than $k$ cycles), divided by 
$n$, tends almost surely to $x_k$ when $n$ goes to infinity.
\end{itemize}
\noindent
This property of convergence of the cycle lengths can easily be translated into an almost sure
convergence of the point process of the eigenangles of the corresponding permutation matrices, if 
the angles associated to the permutation $\sigma_n$ are multiplied by a factor $n$. 
In this way, one obtains a deterministic limiting point process, containing the multiples of 
$2 \pi/ x_k$ for all $k \geq 1$. This property of almost sure convergence can be extended to the 
general case of central measures: in this case, the limiting point process is random, and 
the sequence $(x_k)_{k \geq 1}$ of asymptotic cycle lengths follows the distribution $\nu$ defined 
by \eqref{nu}. An interesting family of central measures is obtained by taking, for a given 
parameter $\theta \in \mathbb{R}_+^*$, $\nu$ equal to the law of a Poisson-Dirichlet process 
of parameter $\theta$. In this case, $\mu$ is the so-called {\it Ewens measure of parameter $\theta$}, i.e. the unique measure under which for all $n \geq 1$,
the coordinate $\sigma_n$ satisfies the following: for all $\sigma \in \mathcal{S}_n$,
$$\mathbb{P}_{\mu} [\sigma_n= \sigma] = \frac{\theta^{k_{\sigma}}}{\theta(\theta+1)\dots(\theta+n-1)},$$
where $k_{\sigma}$ denotes the number of cycles of $\sigma$. In particular, for $\theta =1$, $\mu$ is the 
{\it Haar measure} on $\mathcal{S}^{\infty}$, i.e. the unique measure 
such that for all $n \geq 1$, the $n$-th projection $\sigma_n$ follows the uniform measure on $\mathcal{S}_n$ under $\mu$. 
A random virtual permutation following the Ewens measure of parameter $\theta$ can be constructed by the Chinese
 restaurant process in a convenient way: conditionally
on $\sigma_n$, $n+1$ is a fixed point of $\sigma_{n+1}$ with probability $\theta/(\theta + n)$, otherwise,
it is inserted inside the cycle structure of $\sigma_n$,
each of the $n$ possible places having the same probability $1/(\theta+n)$.

 A similar study has been made by Olshanski and Vershik \cite{OV}, for the space $\mathcal{H}$ of 
 the infinite-dimensional hermitian matrices, i.e. 
 the families $(M_{ij})_{i,j \geq 1}$ of complex numbers such that 
 $M_{ij} = \overline{M_{ji}}$ for all $i, j \geq 1$. A {\it central measure} on $\mathcal{H}$, endowed with the $\sigma$-algebra 
 generated by the coordinates $M_{ij}$, $i, j \geq 1$, is defined as a probability measure $\mu$ such that 
 for all $n \geq 1$, the image of $\mu$ by the projection $(M_{ij})_{i,j \geq 1} \mapsto (M_{ij})_{1 \leq i, j \leq n}$ from $\mathcal{H}$ to 
 the space of $n \times n$ hermitian matrices is invariant by conjugation with any $n \times n$ unitary matrix. In \cite{OV}, it is proven that 
 there exists a family $(\mu^x)_{x \in \Delta}$ of particular measures, again called {\it ergodic measures}, indexed by 
the set 
\begin{align*}
\Delta = \{ & x = (\gamma_1, \gamma_2, x^+_1, x^+_2, \dots, x^+_k, \dots, x^-_1, x^-_2, \dots, x^-_k, \dots); 
\gamma_1 \in \mathbb{R}, \gamma_2 \in \mathbb{R}_+, \\ & \; x^+_1 \geq x^+_2 \geq \dots \geq x^+_k \geq \dots \geq 0, \, 
 x^-_1 \geq x^-_2 \geq \dots \geq x^-_k \geq \dots \geq 0, \\ & (x^+_1)^2 + (x^+_2)^2 + \dots + (x^+_k)^2 + \dots + 
 (x^-_1)^2 + (x^-_2)^2 + \dots + (x^-_k)^2 + \dots  < \infty\},
 \end{align*}
 such that every central measure $\mu$ can be written as follows: 
\begin{equation}
\mu = \int_{\Delta} \mu^x d \nu(x), \label{nu2}
\end{equation}
where $\nu$ is a probability measure on $\Delta$. Moreover, for all 
$$x = (\gamma_1, \gamma_2, x^+_1, x^+_2, \dots, x^+_k, \dots, x^-_1, x^-_2, \dots, x^-_k, \dots) \in \Delta,$$ the ergodic measure 
$\mu^x$ can be characterized by its Fourier transform: for all $n \geq 1$ and for any $n \times n$ hermitian matrix $A$, 
$$\int_{\mathcal{H}} e^{i \operatorname{Tr} [A \, (M_{ij})_{1 \leq i,j \leq n}]} d \mu^x( M) = 
e^{ \, i \gamma_1\operatorname{Tr}(A) \,- \, \gamma_2 \operatorname{Tr}(A^2)/2} \operatorname{det} \left[
\left( \prod_{k=1}^{\infty} \frac{e^{-i x^+_kA}}{1-i x^+_k A} \right) \left( \prod_{k=1}^{\infty} \frac{e^{i x^-_kA}}{1+i x^-_k A} \right) \right]$$
(in this paper, the multiples of identity matrices will sometimes be denoted by complex numbers, and the multiplications by inverse of matrices 
can be denoted as quotients, when there is no problem of commutativity). 
Moreover, if $M$ is a random element of $\mathcal{H}$ following the distribution $\mu^x$, if
 $\lambda_1^+(n) \geq \lambda_2^+(n) \geq \dots \geq \lambda_k^+ (n) \geq \dots \geq 0$ denotes the sequence of positive 
 eigenvalues of the hermitian matrix $(M_{ij})_{1 \leq i, j \leq n}$ and if $\lambda_1^-(n) \leq \lambda_2^-(n) \leq \dots \leq \lambda_k^- (n) \leq \dots
 \leq 0$ is the sequence of negative eigenvalues, both sequences being completed by zeros, then the following properties hold almost surely:
 \begin{itemize}
 \item For fixed $k \geq 1$ and $n$ going to infinity, 
 $$\frac{\lambda^+_k(n)}{n} \longrightarrow x^+_k \; \operatorname{and} \; \frac{\lambda^-_k(n)}{n} \longrightarrow -x^-_k.$$
 \item For $n$ going to infinity, 
 $$\frac{1}{n}\sum_{k=1}^{\infty} (\lambda^+_k(n) + \lambda^-_k(n)) \longrightarrow \gamma_1$$
 and 
 $$\frac{1}{n^2} \sum_{k=1}^{\infty} ((\lambda^+_k(n))^2 + (\lambda^-_k(n))^2) \longrightarrow \gamma_2 + \sum_{k=1}^{\infty} [(x^+_k)^2 + (x^-_k)^2].$$
 \end{itemize}
\noindent
Similarly as in the setting of virtual permutations, the first property implies an almost sure convergence for the renormalized extreme eigenvalues of the left-upper blocks of any element 
in $\mathcal{H}$ following a central measure. In \cite{BO}, Borodin and Olshanski construct a remarkable family of central measures, called {\it Hua-Pickrell measures} and indexed 
by a complex parameter $\delta$ whose real part is strictly larger than $-1/2$. The Hua-Pickrell measure $m^{(\delta)}$ of parameter $\delta$ is defined as the unique 
probability measure such that for all $n \geq 1$, the projection $m^{(\delta,n)}$ of $m^{(\delta)}$ on the space of $n \times n$ hermitian matrices satisfies:
$$m^{(\delta,n)} (dM) = c_{\delta,n} \operatorname{det} ((1+iM)^{-\delta-n}) \operatorname{det} ((1-iM)^{-\bar{\delta}-n}) 
\prod_{1 \leq j \leq k \leq n} d \Re (M_{jk}) \prod_{1 \leq j < k \leq n}  d \Im (M_{jk}) $$
where $c_{\delta,n}$ is a normalization constant. As stated above, the measure $m^{(\delta)}$ can be expressed by 
equation \eqref{nu2}. Moreover, Borodin and Olshanski have  
proven that under $\nu$, $\gamma_2 = 0$ almost surely, and that the point process
 $\{x^+_1, x^+_2, \dots, x^+_k, \dots, -x^-_1, -x^-_2, \dots, -x^-_k, \dots\}$ is a determinantal process whose kernel is explicitly expressed in terms of confluent 
 hypergeometric functions. 
 
 The similarity between the setting of virtual permutations and the setting of infinite hermitian matrices becomes clearer when one replaces 
 hermitian matrices by unitary matrices, via Cayley transform. More precisely, for all $n \geq 1$, the map
 $$\mathcal{C}_n : M \mapsto \frac{M-i}{M+i}$$
 defines a bijection from the set of $n \times n$ hermitian matrices to the set $V(n)$ of $n \times n$ unitary matrices which do not have  $1$ as an eigenvalue, 
 and the inverse bijection is given by
 $$\mathcal{C}_n^{-1} : u \mapsto i \frac{1+u}{1-u}.$$
For $n \geq m \geq 1$, one defines a natural projection $p_{n,m}$ from the space of $n \times n$ hermitian matrices to the
 space of $m \times m$ hermitian matrices, simply by taking the left upper block. This projection $p_{n,m}$ induces a map
 $\tilde{\pi}_{n,m}$ from $V(n)$ to $V(m)$, given by
 $$\tilde{\pi}_{n,m} = \mathcal{C}_m \circ p_{n,m} \circ \mathcal{C}_n^{-1},$$
 and it is immediate to check that $\tilde{\pi}_{n,p} = \tilde{\pi}_{m,p} \circ \tilde{\pi}_{n,m}$ for $n \geq m \geq p \geq 1$. 
Moreover, Neretin \cite{Ner} has computed explicitly the projection $\tilde{\pi}_{n,m}$: if a matrix 
$$u =\left( \begin{array}{cc}
A & B  \\
C & D   \end{array} \right) \in V(n)$$
 is divided into four blocks 
of size $m \times m$, $m \times (n-m)$, $(n-m) \times m$ and $(n-m) \times (n-m)$, then $1-D$ is invertible and 
$$\tilde{\pi}_{n,m}(u) = A + B (1-D)^{-1}C \in V(m).\footnote[1]{In fact, the projections defined here are not exactly those which are given in \cite{Ner}: we have reversed the order of 
the lines and the columns, and we have changed the matrices to their opposite. These slight changes are made for the coherence of the present paper.} $$
 Then, one can define the {\it virtual isometries} (or virtual rotations) as the sequences $(u_n)_{n \geq 1}$ of unitary matrices, such that $u_n \in V(n)$
for all $n \geq 1$ and $u_m = \tilde{\pi}_{n,m} (u_n)$ for all $n \geq m \geq 1$. An equivalent condition is the following: there exists an infinite 
hermitian matrix $M$ such that for all $n \geq 1$, 
$$u_n = \mathcal{C}_n ( (M_{ij})_{1 \leq i \leq j \leq n}).$$
 If the space of virtual isometries is endowed with the $\sigma$-algebra generated by the coordinates, then the Cayley transform induces a bijection between 
 the probability measures on this space and the probabilities on the space of infinite hermitian matrices, and also a bijection between the central measures:
 hence, the results of Olshanski and Vershik \cite{OV} can be immediately translated into a classification of all the central measures on the 
 space of virtual isometries. One also deduces that under such a central measure, and for all $k \geq 1$, the $k$-th smallest positive and negative eigenangles 
 of $u_n$, multiplied by $n$, tend almost surely to a limit when $n$ goes to infinity. The translation of the particular case of Hua-Pickrell measures into the 
 unitary context gives a family $\tilde{\mu}^{(\delta)}$ of central measures on the space of virtual rotations, such that for all $n \geq 1$, its $n$-th projection
 $\tilde{\mu}^{(\delta,n)}$ is given by 
$$\tilde{\mu}^{(\delta,n)} (du) = c'_{\delta,n} \det(1-u)^{\bar{\delta}}\det(1-\overline{u})^{\delta} du,$$
where $c'_{\delta,n}$ is a normalization constant, and $du$ the Haar measure on $V(n)$, i.e. the restriction to $V(n)$ of the Haar measure on the unitary group $U(n)$, which is 
a probability measure since $V(n)$ contains almost every element of $U(n)$. In the case where $\delta = 0$, the measure 
$\tilde{\mu}^{(0)}$ can be called {\it Haar measure} on the space of virtual isometries, since for all $n \geq 1$, its $n$-th projection 
is equal to the Haar measure 
on $V(n)$. Moreover, the corresponding limiting point process of the renormalized scaled eigenangles is a determinantal process with 
sine kernel: informally, if $x_1, \dots, x_k \in \mathbb{R}$, then the probability that there is
 a point in the neighborhood of $x_j$ for all $j \in \{1,\dots,k\}$ is
 proportional to $\det(K(x_j,x_l)_{1 \leq j,l \leq k})$, where the kernel $K$ is given by
 $$K(x,y) =\frac{1}{2\pi} \sin(\pi(x-y))/(\pi(x-y).$$ 
 Note that the weak version of this result (convergence in law of the point process of the eigenangles towards a sine kernel process) is very classical in random 
 matrix theory. The introduction of the virtual rotations for this problem has the following advantages:
  \begin{itemize}
 \item One has a result of almost sure convergence when the dimension goes to infinity, which is quite rare in random matrix theory. 
 \item The limit point process is directly associated to a random virtual isometry, with a deterministic map (almost surely well-defined): the link
 between sine kernel process and random matrices is particularly explicit in this setting. 
 \end{itemize}
 \noindent
 As we have seen in this introduction, there are two kinds of sequences $(u_n)_{n \geq 1}$ of unitary matrices giving a similar behavior for the 
 small eigenangles: the virtual permutations, for which $u_n$ is a
 $n \times n$ permutation matrix (identified with an element of $\mathcal{S}_n$) for all $n \geq 1$, and the virtual rotations, for which 
 $u_n \in V(n)$. 
 
 This work is intended as an attempt to understand the links between the virtual permutations  and the virtual rotations.  Given that the group of permutations of size $n$ can be identified with the corresponding subgroup of permutation matrices of the unitary group, it is natural to expect that  virtual permutations can be obtained from the construction of virtual rotations. However it turns out that one cannot recover the virtual permutations from Neretin's construction because in this construction, for any $n$, $1$ cannot be an eigenvalue of $u_n$ for any virtual rotation $(u_n)_{n\geq1}$. We hence propose another construction for virtual isometries, based on complex reflections, which extends both constructions. Our construction provides us with a simple recursive relation between the characteristic polynomials of $u_n$ and $u_{n+1}$ for a virtual isometry $(u_n)_{n\geq1}$. Moreover it shows us how to generate random virtual isometries. As a consequence we are able to recover the known result that  the smallest normalized eigenangles of a virtual rotation, under the Haar measure, converge a.s. to a point process
whose correlation function is given by the sine kernel. Our proof exhibits an interesting interlacement properties for the eigenangles of $u_n$ and $u_{n+1}$.

More precisely,  in Section \ref{Vi}  we define a projection from $U(n)$ to $U(m)$, which  extends both the projection 
 $\pi_{n,m}$ from $\mathcal{S}_n$ to $\mathcal{S}_m$ and the projection $\tilde{\pi}_{n,m}$ from $V(n)$ to $V(m)$. 
In Section \ref{Haar2}, we deduce in our framework a natural construction of the Hua-Pickrell measures (and in particular the Haar measure) on virtual rotations, 
in terms of products of independent random reflections. As we have seen before, the general results by Borodin, Olshanski and Vershik given in \cite{BO} and \cite{OV} imply
the almost sure convergence of the 
renormalized eigenangles, for a virtual isometry following Haar measure. In Section \ref{conv}, we give a direct and purely probabilistic proof of this result, which 
gives some information on the corresponding rate of convergence.
\section{The space of virtual isometries} \label{Vi}

As stated in the introduction, our purpose in this section is to define a strict analogue of the virtual permutations,
in the context of the unitary group. Our construction is expected to be applicable to other compact groups
(like the orthogonal or the symplectic group), however, for the sake of simplicity, we only deal
with unitary matrices in this article. The first step of our construction is the
following result, which proves that intuitively, it is possible to construct
a natural projection from the unitary group of a finite-dimensional Hilbert space $E$ to a unitary group
of a subspace of $F$.
\begin{prop} \label{projection}
Let $H$ be a complex Hilbert space, $E$ a finite-dimensional subspace of $H$ and $F$ a subspace of $E$. Then, for
 any unitary operator $u$ on $H$ which fixes
each element of $E^{\perp}$, there exists a unique unitary operator $\pi_{E,F}(u)$ on $H$ which satisfies
the  following two conditions:
\begin{itemize}
\item $\pi_{E,F}(u)$ fixes each element of $F^{\perp}$;
\item the image of $H$ by $u-\pi_{E,F}(u)$ is included into the image of $F^{\perp}$ by $u-\Id$.
\end{itemize}
Moreover, if $G$ is a subspace of $F$, $\pi_{F,G} \circ \pi_{E,F} (u)$ is well-defined and is
equal to $\pi_{E,G} (u)$.
\end{prop}

\begin{proof}
Let $x$ be an element of $F \cap (u-\Id)(F^{\perp})$. There exists $y \in F^{\perp}$ such that
$x = u(y) - y$, or equivalently, $u(y) = x + y$. Since $x \in F$
and $y \in F^{\perp}$, one has $||x||^2 + ||y||^2 = ||x+y||^2 = ||u(y)||^2 = ||y||^2$, which implies
$x = 0$. Now, if $v_1$ and $v_2$ are two operators which satisfy
the conditions defining  $\pi_{E,F}(u)$, one has the following:
\begin{itemize}
\item $v_1$ and $v_2$ fix globally the space $F$, since they fix the space $F^{\perp}$;
\item $v_1 - v_2$ vanishes on $F^{\perp}$, since $v_1$ and $v_2$ fix each point of this space;
\item the image of $v_1 - v_2$ is included in $(u-\Id)(F^{\perp})$, since the images of $u-v_1$ and
$u-v_2$ are both included in this space.
\end{itemize}
These three properties imply that the image of $v_1-v_2$ is included in the space
$F \cap (u-\Id)(F^{\perp}) = \{0\}$, which
proves the uniqueness of $\pi_{E,F} (u)$. Let us now show its existence and the projective property of the
map $\pi_{E,F}$.

We can first remark that for $G \subset F \subset E \subset H$,
 if $\pi_{E,F}$ and $\pi_{F,G}$ are well-defined, then $\pi_{E,G}$ is also well-defined and
is equal to $\pi_{F,G} \circ \pi_{E,F}$. Indeed, for all unitary operators $u$
fixing each element of $E^{\perp}$, the two unitary operators
$v := \pi_{E,F}(u)$ and $w := \pi_{F,G}(v)$ are well-defined and satisfy the following assumptions:
\begin{itemize}
\item $v$ fixes all the elements of $F^{\perp}$;
\item $(u-v)(H) \subset (u-\Id)(F^{\perp})$;
\item $w$ fixes all the elements of $G^{\perp}$;
\item $(v-w)(H) \subset (v-\Id)(G^{\perp})$.
\end{itemize}
\noindent
This yields the elementary inclusions
\begin{align*}
(u -  w)(H)  & \subset \Span((u-v)(H), (v - w)(H)) \\ &
 \subset  \Span((u-\Id)(F^{\perp}),(v-\Id)(G^{\perp})) \\ &
\subset \Span ((u-\Id)(F^{\perp}), (u-\Id)(G^{\perp}), (u - v)(G^{\perp}))   \\ &
\subset \Span((u-\Id)(G^{\perp}), (u - v)(H))\\& \subset (u-\Id) (G^{\perp}).
\end{align*}
\noindent
Since $w$ fixes each element of $G^{\perp}$, $\pi_{E,G}(u)$ is well-defined and is equal to $w$. By
 induction, it is now sufficient to prove the existence
of $\pi_{E,F}$ in the particular case where $E= \operatorname{Vect}(F,e)$, where $e$ is a unit
 vector, orthogonal to $F$. In this case, if $u$ is a unitary operator fixing each element of
$E^{\perp}$, then the operator $v := \pi_{E,F}(u)$
can be constructed explicitly as follows.
\begin{itemize}
\item If $u(e) = e$, then one takes $v = u$, whch fixes $E^{\perp}$ and $e$, hence, it fixes $F^{\perp}$, and
$(u-v)(H) = \{0\} = (u-\Id)(F^{\perp})$.
\item If $u(e) \neq e$, then for all $x \in H$, we define
\begin{equation}
v(x) := u(x) + \frac{ \langle e - u(e), u(x) \rangle}{\langle e - u(e), u(e) \rangle} \, (e - u(e)). \label{v}
\end{equation}
The denominator in the expression defining $v(x)$ does not vanish and the following properties hold:
\begin{itemize}
\item for all $x \in H$,
\begin{align*} ||v(x)||^2 & = ||u(x)||^2 + 2 \Re \left( \frac{| \langle e- u(e), u(x) \rangle |^2}{ \langle e - u(e), u(e) \rangle} \right)
+ \frac{| \langle e- u(e), u(x) \rangle |^2}{| \langle e - u(e), u(e) \rangle|^2} \, ||e-u(e)||^2 \\
& = ||x||^2 + \left( \frac{| \langle e- u(e), u(x) \rangle |^2}{| \langle e - u(e), u(e) \rangle|^2} \right) \,
\left( 2 \Re ( \langle e - u(e), u(e) \rangle) + ||e-u(e)||^2 \right)
\\ &=  ||x||^2 + \left( \frac{| \langle e- u(e), u(x) \rangle |^2}{| \langle e - u(e), u(e) \rangle|^2} \right) \, (||e||^2 - ||u(e)||^2)\\
& = ||x||^2,
\end{align*}
which implies that $v$ is a unitary operator;
\item for $x \in E^{\perp}$, $u(x) = x$, and $e - u(e) \in E$, since $E$ is globally fixed by $u$. \\ Hence, $\langle e - u(e), u(x) \rangle = 0$, and
$v(x) = u(x) = x$: $v$ fixes each element of $E^{\perp}$;
\item by \eqref{v}, $v$ fixes $e$, and then, it fixes each element of $F^{\perp}$;
\item Again by \eqref{v}, for all $x \in H$, $u(x)-v(x)$ is a multiple of $u(e) - e$, and then in the image of $F^{\perp}$ by $u - \Id$.
\end{itemize}
\end{itemize}
This concludes the proof.
\end{proof}

\begin{rem}
In the case where $E= \operatorname{Vect}(F,e)$, $e$ being a unit
 vector, orthogonal to $F$, the formula \eqref{v} proves that for $e \neq u(e)$,
$u=r \pi_{E,F}(u)$, where $r$ is the unique reflection (i.e. $r$ is unitary and $r - \Id$ has rank one)
such that $r(e) =u(e)$. Similarly, $u = \pi_{E,F}(u) r'$ where $r'$ is the unique reflection such that
$r'(u^{-1}(e)) = e$.
\end{rem}
The following consequence of Proposition \ref{projection} shows that $\pi_{E,F}$ is also a projection
in the sense of the minimization of a distance:
\begin{prop} \label{projection2}
Let $H$ be a complex Hilbert space, and let $U_0(H)$ be the space of the unitary operators on $H$ which fix
each element of the orthogonal of a finite-dimensional subspace of $H$. Then the map
$\dd$ from $U_0(H) \times U_0(H)$ to $\mathbb{N}_0$, given by $\dd(u,v) := \rank(u-v)$ defines a finite
distance on $U_0(H)$. Moreover, if $F \subset E$ are two finite-dimensional subspaces of $H$, and if $u$
is a unitary operator fixing each element of $E^{\perp}$, then $\pi_{E,F}(u)$ is the unique
unitary operator fixing each element of $F^{\perp}$ and such that $\dd(u, \pi_{E,F}(u))$ is minimal.
The image of $H$ by $u - \pi_{E,F}(u)$ is equal to the image of $F^{\perp}$ by $u -\Id$, and one has:
$$ \dd(u, \pi_{E,F}(u)) = \dim(E) - \dim(F) - \dim(\{x \in E \cap F^{\perp}, u(x) = x\}),$$
in particular, if one is not an eigenvalue of the restriction of $u$ to $E \cap F^{\perp}$, then
$$ \dd(u, \pi_{E,F}(u)) = \dim(E) - \dim(F).$$
\end{prop}
\begin{proof}
Let $u, v \in U_0(H)$. By assumption, the images of $u - \Id$ and $v - \Id$ are finite-dimensional,
and then $u-v$ has finite rank: $\dd(u,v)$ is finite. It is obvious that $\dd(u,u)=0$ and
$\dd(u,v) = \dd(v,u)$, and if $\dd(u,v) = 0$, the image of $u-v$ is equal to $\{0\}$,
which implies $u=v$. Moreover, if $w \in U_0(H)$, then
$$(u-w) (H) \subset \Span((u-v)(H), (v-w)(H)),$$
which implies
$$\dd(u,w) \leq \dd(u,v) + \dd(v,w).$$
Hence $\dd$ defines a finite distance. Now, let us suppose that $u$ fixes each element in $E^{\perp}$
and $v$ fixes each element in $F^{\perp}$, for two finite-dimensional spaces $F \subset E$.
Then, for $x \in F^{\perp}$, $(u-v)(x) = (u-\Id)(x)$, which implies that
the image of $u-v$ contains the image of $F^{\perp}$ by $u-\Id$. Since by Proposition \ref{projection},
$(u-\pi_{E,F}(u)) (H) \subset (u-\Id)(F^{\perp})$:
\begin{itemize}
\item $(u-\pi_{E,F}(u))(H) = (u-\Id)(F^{\perp})$,
\item $\pi_{E,F}(u)$ is the unique unitary operator $v$, fixing each element of $F^{\perp}$, and
such that the space $(u-v)(H)$, and then the distance $\dd(u,v)$, is minimal.
\end{itemize}
\noindent
Now, since $u$ fixes each element of $E^{\perp}$, $(u-\Id)(F^{\perp})= (u-\Id) (E \cap F^{\perp})$, and
then
\begin{align*}
 d(u, \pi_{E,F}(u)) & = \dim(E \, \cap \, F^{\perp}) - \dim ( \operatorname{Ker} (u- \Id)
\, \cap \,  E \cap F^{\perp}) \\ & = \dim (E) - \dim(F) - \dim(\{x \in E \, \cap F^{\perp}, u(x) = x\}),
\end{align*}
which concludes the proof.
\end{proof}
\noindent
\begin{rem} For other distances $\dd'$ on $U_0(H)$, for $F \subset E \subset H$, $E$ finite-dimensional, and
for a unitary operator $u$ fixing all the elements of $E^{\perp}$, it can be possible to define
$\pi^{(\dd')}_{E,F}(u)$ as the unitary operator $v$ which fixes the elements of $F^{\perp}$ and
for which the distance $\dd'(u,v)$ is as small as possible. However, for $G \subset E$, one does not have in general
$\pi^{(\dd')}_{E,G} = \pi^{(\dd')}_{F,G} \, \pi^{(\dd')}_{E,F}$. A natural question, not treated here, is the following: what
 are the distances $\dd'$ for which the projective property
$\pi^{(\dd')}_{E,G} = \pi^{(\dd')}_{F,G} \, \pi^{(\dd')}_{E,F}$ remains true?
\end{rem}

The existence of the projective map described above implies the possibility to define the virtual isometries.
Indeed, let $H := \ell^2(\mathbb{C})$, and let $(e_n)_{n \geq 1}$ be the canonical
Hilbert basis of $H$. For all $n \geq 1$, the space of unitary operators fixing each
element of the orthogonal of $\Span(e_1,\dots,e_n)$ can be canonically identified with
the unitary group $U(n)$. By this identification, for $n \geq m \geq 1$, the projection
$\pi_{\Span(e_1,\dots,e_n), \Span(e_1,\dots,e_m)}$ defines a map $\pi_{n,m}$ from
$U(n)$ to $U(m)$, and for $n \geq m \geq p \geq 1$, one has
$\pi_{n,p} = \pi_{m,p} \circ \pi_{n,m}$.

\begin{defn} \label{casp}
A {\it virtual isometry} is a
 sequence $(u_n)_{n \geq 1}$ of unitary matrices, such that for all
$n \geq 1$, $u_n \in U(n)$ and $\pi_{n+1,n}(u_{n+1}) = u_n$. In this case, for all $n \geq m \geq 1$,
$\pi_{n,m}(u_n) = u_m$. The space of virtual isometries will be denoted $U^{\infty}$.
\end{defn}
\begin{rem}
Given two virtual isometries $(u_n)_{n \geq 1}$ and $(v_n)_{n \geq 1}$, the sequence $(w_n)_{n \geq 1}$ obtained from pointwise multiplication $w_n=u_n v_n$
 is not a virtual isometry in general: the coherence property doesn't hold for $(w_n)_{n \geq 1}$. Hence $U^\infty$ has no group structure.
\end{rem}
\noindent
It is now possible to check that the virtual isometries defined in the present paper are both a generalization of the 
virtual permutations, and an extension of the virtual isometries in the sense of Neretin. 
\begin{prop}
Let $(\sigma_n)_{n \geq 1}$ be a sequence of permutations such that $\sigma_n \in \mathcal{S}_n$ for all
$n \geq 1$, and let $(\Sigma_n)_{n \geq 1}$ be the corresponding sequence of permutation matrices. Then
$(\Sigma_n)_{n \geq 1}$ is a virtual isometry if and only if $(\sigma_n)_{n \geq 1}$ is a virtual
permutation.
\end{prop}
\begin{proof}
In this proof, and in all the sequel of the article, we identify $\mathbb{C}^n$ with the
set of the sequences $(x_k)_{k \geq 1}$ such that $x_k = 0$ for all $k >n$, and we define
$(e_k)_{k \geq 1}$ as the canonical basis of $\mathbb{C}^{\mathbb{N}}$: in this way,
$(e_k)_{1 \leq k \leq n}$ is identified with the canonical basis of $\mathbb{C}^n$ for all $n \geq 1$.
With this convention, the sequence $(\Sigma_n)_{n \geq 1}$ is a virtual
isometry if and only if for all $n \geq 1$, the
image of $\Sigma_{n+1} - \Sigma_n$ is in the vector space generated by $e_{n+1} - \Sigma_{n+1}(e_{n+1})
= e_{n+1} - e_{\sigma_{n+1}(n+1)}$. Since $\Sigma_{n+1}(e_j) - \Sigma_n(e_j) =
 e_{\sigma_{n+1}(j)}- e_{\sigma_n(j)}$,
the condition above is satisfied if and only if for all $n \geq 1$, $j \in \{1,\dots,n\}$, one
 of the two following situations arises:
\begin{itemize}
\item $\sigma_{n+1}(j) = \sigma_n(j)$;
\item $\sigma_{n+1}(j) = n+1$ and $\sigma_n(j) = \sigma_{n+1}(n+1)$.
\end{itemize}
\noindent
In other words, $(\Sigma_n)_{n \geq 1}$ is a virtual isometry if and only if for all
$n \geq 1$, one of the two following cases holds:
\begin{itemize}
\item the restriction of $\sigma_{n+1}$ to $\{1,\dots,n\}$ is equal to $\sigma_n$, and
 $\sigma_{n+1}(n+1) = n+1$;
\item $\sigma_{n+1}(n+1) \neq n+1$,
$\sigma_{n}(j) = \sigma_{n+1}(j)$ for $j \in \{1, \dots, n\} \backslash \{\sigma_{n+1}^{-1}(n+1)\}$, and
$\sigma_{n}(j) = \sigma_{n+1}(n+1)$ for $j = \sigma_{n+1}^{-1}(n+1)$.
\end{itemize}
\noindent
This is equivalent to the fact that $(\sigma_n)_{n \geq 1}$ is a virtual permutation.
\end{proof}
\noindent
\begin{prop}
Let $(u_n)_{n \geq 1}$ be a sequence of unitary matrices such that $u_n \in V(n)$ for all $n \geq 1$ (recall that 
$V(n)$ is the set of $n \times n$ unitary matrices which do not have one as an eigenvalue).
Then, $(u_n)_{n \geq 1}$ is a virtual rotation in the sense of Neretin if and only if it is a virtual isometry in the 
sense of Definition \ref{casp}. 
\end{prop}
\begin{proof}
It is sufficient to check that for $n \geq m \geq 1$, and $u_n \in V(n)$, $\pi_{n,m}(u_n) = \tilde{\pi}_{n,m} (u_n)$, 
where $\tilde{\pi}_{n,m}$ is defined as in the introduction. 
By Proposition \ref{projection2}, one deduces that it is sufficient to bound the 
rank of 
$$R := u_n - \left( \begin{array}{cc}
\tilde{\pi}_{n,m} (u_n) & 0  \\
0 & \Id_{n-m}   \end{array} \right) $$
by $n-m$, since one is not an eigenvalue of $u_n$. 
Now, if $u_n$ is divided into blocks of size $m \times m$, $m \times (n-m)$, $(n-m) \times m$, $(n-m) \times (n-m)$:
$$u_n = \left( \begin{array}{cc}
A& B  \\
C &D   \end{array} \right),$$
then
\begin{align*} R & = \left( \begin{array}{cc}
A& B  \\
C &D   \end{array} \right) - \left( \begin{array}{cc}
A+ B (1-D)^{-1}C & 0  \\ 0 & 1  \end{array} \right)
\\ & = \left( \begin{array}{cc}
B (D-1)^{-1}C& B  \\
C &D-1   \end{array} \right) = 
 \left( \begin{array}{c}
B (D-1)^{-1} \\
1   \end{array} \right) \left( \begin{array}{cc}
C & D-1   \end{array} \right). 
\end{align*}
In other words, $R$ is the product of a $n \times (n-m)$ matrix by a $(n-m) \times n$ matrix: its rank cannot be strictly larger than $n-m$.
\end{proof}
\noindent
Now, since the virtual isometries are the natural generalizations of the virtual permutations,
it is natural to ask if there is an analog of the Chinese restaurant process. The answer is
positive:
\begin{prop} \label{chinois}
Let $(x_n)_{n \geq 1}$ be a sequence of vectors, $x_n$ lying on the complex unit sphere of
$\mathbb{C}^n$ for all $n \geq 1$. Then, there exists a unique virtual isometry $(u_n)_{n \geq 1}$
such that $u_n(e_n) = x_n$ for all $n \geq 1$, and $u_n$ is given by
$$u_n = r_n r_{n-1}\dots r_1,$$
 where for $j \in \{1, \dots,n\}$, $r_j = \Id$ if $x_j = e_j$, and otherwise, $r_j$ is the unique
 reflection such that
 $r_j (e_n) = x_n$. Moreover, in the particular case where for all $n \geq 1$, $x_n = e_{i_n}$ for
 $i_n \in \{1, \dots, n\}$, then $(u_n)_{n \geq 1}$ is the sequence of matrices associated to a
virtual permutation $(\sigma_n)_{n \geq 1}$ constructed by the Chinese restaurant process: for all
$n \geq 1$,
$$\sigma_n = \tau_{n,i_n} \tau_{n-1, i_{n-1}} \dots \tau_{1,i_1},$$
where, for $j, k \in \{1,\dots,n\}$, $\tau_{j,k}=\Id$ if $j=k$ and $\tau_{j,k}$ is the transposition
$(j,k)$ if $j \neq k$.
\end{prop}
\begin{proof}
One has $u_1(e_1) = x_1$ if and only if $u_1 = x_1$, which is equal to $r_1$. For all $n \geq 1$,
two cases are possible:
\begin{itemize}
\item if $x_{n+1} = e_{n+1}$, then $\pi_{n+1,n}(u_{n+1}) = u_n$ and $u_{n+1}(e_{n+1})= e_{n+1}$ if and
only if $u_{n+1} = (u_{n}) \oplus 1$, where the symbol $\oplus$ denotes diagonal blocks of matrices;
\item if $x_{n+1} \neq e_{n+1}$, the equation \eqref{v} and the remark after the proof of
Proposition \ref{projection} imply that
$\pi_{n+1,n}(u_{n+1}) = u_n$ and $u_{n+1}(e_{n+1})=x_{n+1}$ if and only if $u_{n+1} = r_{n+1} \,
(u_n \oplus 1)$.
\end{itemize}
\noindent
By induction, the uniqueness and the general form of $u_n$ is proven. If $x_n = e_{i_n}$ for all $n \geq 1$,
$r_n$ is the matrix of the permutation $\tau_{n,i_n}$, which easily implies the second part of
Proposition \ref{chinois}.
\end{proof}
\noindent
The construction given in Proposition \ref{chinois} implies in particular that the space $U^{\infty}$ is
 not empty. Moreover, it is possible to use it to define probability measures on this space.

\section{Some remarkable measures on $U^{\infty}$} \label{Haar2}
Once the space $U^{\infty}$ is constructed, it is natural to ask if there exists an analog of the
Haar measure on this space. As seen in the introduction, the positive answer can be deduced from the results given 
in \cite{BO}, \cite{Ner} and \cite{OV}, and of the fact that under Haar measure on $U(n)$, almost every matrix is in $V(n)$. 
A more direct proof can be easily deduced from the results given by Bourgade, Nikeghbali and Rouault in 
 \cite{BNR}:
\begin{prop} \label{H}
Let $(x_n)_{n \geq 1}$ be a random sequence of vectors, $x_n$ lying on the complex unit sphere of
$\mathbb{C}^n$ for all $n \geq 1$, and let $(u_n)_{n \geq 1}$ be the unique virtual isometry
such that $u_n(e_n)= x_n$ for all $n \geq 1$. Then, for each $n$, the matrix $u_n$ follows the
Haar measure on $U(n)$ if and only if $x_1,\dots,x_n$ are independent and for all $j \in \llbracket 1,n\rrbracket$,
$x_j$ follows uniform measure on the complex unit sphere of $\mathbb{C}^n$.
\end{prop}
\noindent
This result has the following consequence, showing the compatibility between the Haar measure on
 $U(n)$, $n \geq 1$ and the projections $\pi_{n,m}$, $n \geq m \geq 1$.
\begin{prop} \label{Haar}
For all $n \geq m \geq 1$, the image of the Haar measure on $U(n)$ by the projection $\pi_{n,m}$ is equal to
the Haar measure on $U(m)$.
\end{prop}
\begin{rem}
The statement given in Proposition \ref{Haar} is meaningful only if the application $\pi_{n,m}$ is
measurable with respect to the Borel $\sigma$-algebras of $U(n)$ and $U(m)$. This fact can be easily checked
by using the formula \eqref{v}. Moreover, Proposition \ref{Haar} can be easily proven
 directly. Indeed, let $n \geq 1$,
let $u$ be a matrix on $U(n+1)$ following the Haar measure, and let $a$ be a deterministic matrix
on $U(n)$. The invariance of the Haar measure implies that $u(a \oplus 1)$ follows the Haar measure
on $U(n+1)$, and then has the same law as $u$. Now, it is easy to check that $\pi_{n+1,n}( u(a \oplus 1) )
 =\pi_{n+1,n}(u)\, a$, hence, $\pi_{n+1,n}(u) \, a$ has the same law as $\pi_{n+1,n}(u)$ for all $a \in U(n)$.
One deduces that $\pi_{n+1,n} (u)$ follows the Haar measure on $U(n)$.
\end{rem}
\noindent
The property of compatibility given in Proposition \ref{Haar} implies
the possibility to define the Haar measure on $U^{\infty}$. In order to do this properly, let us prove the
following result, about the extension of measures:
\begin{prop} \label{extension}
Let $\mathcal{U}$ be the $\sigma$-algebra on $U^{\infty}$, generated by the sets:
$$\{(u_n)_{n \geq 1}, u_k \in B_k \},$$
for all $k \geq 1$ and for all Borel sets $B_k$ in $U(k)$. Let $(\mu_n)_{n \geq 1}$ be a
family of probability measures, $\mu_n$ defined on the space $U(n)$ (endowed with its Borel $\sigma$-algebra),
and such that the image of $\mu_{n+1}$ by $\pi_{n+1,n}$ is equal to $\mu_n$ for all $n \geq 1$. Then,
there exists a unique probability measure on $(U^{\infty}, \mathcal{U})$ such that its image by the
$n$-th coordinate is equal to $\mu_n$ for all $n \geq 1$.
\end{prop}
\begin{proof}
For any element $(u_n)_{n \geq 1}$ in $U^{\infty}$, $u_p$ can be expressed
as a Borel function of $u_m$ for all $m \geq p \geq 1$. One deduces that the family of sets
of the form $$\{(u_n)_{n \geq 1}, u_k \in B_k \}$$ is stable by finite intersection. This implies the
uniqueness part of Proposition \ref{extension}, by the monotone class theorem. In order to prove the existence,
let us consider the product $V$ of all the unitary groups $U(n)$, $n \geq 1$, endowed with
the product $\mathcal{V}$ of their Borel $\sigma$-algebras.
For all $n \geq 1$, let
us define the measure $\tilde{\mu}_n$ on the space $U(1) \times \, \dots \, \times U(n)$, endowed
with the corresponding product of Borel $\sigma$-algebras, as the image
of $\mu_n$ by the map:
$$u_n \mapsto (\pi_{n,1}(u_n),\dots,\pi_{n,n-1}(u_n), \pi_{n,n}(u_n)),$$
from $U(n)$ to $U(1) \times \, \dots \, \times U(n)$.
The projective property of $\pi_{m,p}$, $m \geq p \geq 1$ and the fact that $\mu_{n}$ is the image
of $\mu_{n+1}$ by $\pi_{n+1,n}$ implies that for all $n \geq 1$, the restriction of $\tilde{\mu}_{n+1}$ to
the $n$ first coordinates is equal to $\tilde{\mu}_n$. The classical theorem of extension
of probability measures implies that there exists a measure $\tilde{\mu}$ on $(V,\mathcal{V})$
such that its restriction to the $n$ first coordinates is equal to $\tilde{\mu}_n$, for all $n \geq 1$.
By construction, $\tilde{\mu}$ is carried by the set $U^{\infty}$, which is in $\mathcal{V}$, hence, it induces a measure on
$U^{\infty}$, endowed with the intersection of $\mathcal{V}$ and $\mathcal{P}(U^{\infty})$, which is equal to $\mathcal{U}$.
\end{proof}
\noindent
By combining the Propositions \ref{H} and \ref{extension}, one deduces the existence of the
Haar measure on virtual isometries:
\begin{prop} \label{Haarvi}
There exists a unique probability measure $\mu^{(0)}$ on the space $(U^{\infty}, \mathcal{U})$ such that
its image by all the coordinate maps are equal to the Haar measure on the corresponding unitary group.
This measure can be described as follows. Let $(x_n)_{n \geq 1}$ be a random sequence of
 vectors, $x_n$ lying on the complex unit sphere of
$\mathbb{C}^n$ for all $n \geq 1$, and let $(u_n)_{n \geq 1}$ be the unique virtual isometry
such that $u_n(e_n)= x_n$ for all $n \geq 1$. Then, the distribution of $(u_n)_{n \geq 1}$
is equal to $\mu^{(0)}$ if and only if $(x_n)_{n \geq 1}$ are independent, and for all $n \geq 1$,
$x_n$ follows the uniform measure on the complex unit sphere of $\mathbb{C}^n$.
\end{prop}
\noindent
The Haar measure $\mu^{(0)}$ is the
analog of the uniform measure on virtual permutations, which can be obtained in the setting of Proposition
\ref{Haarvi}, by taking $(x_n)_{n \geq 1}$ independent, $x_n$ uniform on the
finite set $\{e_1,\dots,e_n\}$. Moreover, it is possible to generalize the Haar measure on virtual isometries,
in the same way as uniform measure on virtual permutations can be generalized by considering the
Ewens measures. Enouncing this generalization requires the so-called $h$-sampling (or
$h$-transform), which can be
described as follows. Let $(X,\mathcal{F},\mu)$ be a probability space. For a given measurable
function $h:X\mapsto \RR^+$ such that $0<\E_\mu(h)<\infty$, a probability measure $\mu'$
is said to be the $h$-sampling of $\mu$ if and only if for all bounded measurable
functions $f$,
$$
\E_{\mu'}(f)=\frac{\E_\mu(f\,h)}{\E_\mu(h)}.
$$
Here, for all $n \geq 1$, and for $\delta \in \mathbb{C}$ such that $\Re(\delta) > -1/2$, it is possible
 to define a probability measure $\nu_{\delta}^{(n)}$ as the $h$-sampling of the uniform measure
on the complex unit sphere, for
\begin{equation}
h(x)=(1-\langle e_n,x\rangle)^{\overline{\delta}}(1-\overline{\langle e_n,x\rangle})^{\delta}, \label{hhp}
\end{equation}
where the imaginary part of the logarithm of $1-\langle e_n,x\rangle$ is taken in the
interval $(-\pi/2,\pi/2)$. Then, in \cite{BNR}, Bourgade, Nikeghbali and Rouault have essentially
 proven the following result:
\begin{prop} \label{Ewens}
 Let $(x_n)_{n \geq 1}$ be a random sequence of
 independent vectors such that for all $n \geq 1$, $x_n$ follows the distribution
 $\nu_{\delta}^{(n)}$ on the complex unit sphere of
$\mathbb{C}^n$ for all $n \geq 1$, and let $(u_n)_{n \geq 1}$ be the unique virtual isometry
such that $u_n(e_n)= x_n$ for all $n \geq 1$. Then, for all $n \geq 1$, the distribution of
 $u_n$ can be described as the $h$-sampling of the Haar measure on $U(n)$, where the function $h$
is given by
$$
h(u)=\det(\Id-u)^{\bar{\delta}}\det(\Id-\overline{u})^{\delta},
$$
where the logarithm of $\det(\Id -u)$ is taken in the unique way such that it is continuous
on the connected set $\{u \in U(n), \det(\Id-u) \neq 0\}$, and real (equal to $n \log(2)$) for $u = -\Id$.
This construction determines a measure $\mu^{(\delta)}$ on the space $(U^{\infty}, \mathcal{U})$,
 which can be identified with the Hua-Pickrell measure $\tilde{\mu}^{(\delta)}$ given in the introduction.
\end{prop}
\noindent
The Hua-Pickrell measures are also analogs of Ewens measures on the
space of virtual permutations. Indeed, the Ewens measure of parameter $\theta > 0$, on the space
of virtual permutations, can be constructed in our framework by taking
$(x_n)_{n \geq 1}$ independent and for all $n \geq 1$:
\begin{itemize}
\item $x_n \in \{e_1,\dots,e_n\}$ almost surely;
\item $\mathbb{P} [x_n = e_n] = \theta/(\theta + n-1)$;
\item For all $j \in \{1,\dots,n-1\}$, $\mathbb{P}[x_n = e_j] =1/(\theta + n-1)$.
\end{itemize}
\noindent
The law of $x_n$ can be viewed as an $h$-sampling of the uniform measure on
the space $\{e_1,\dots, e_n\}$, where the function $h$ can be written
\begin{equation}
h(x)= (1 + \langle e_n, x\rangle)^{2 \delta}, \label{he}
\end{equation}
for $\delta = \log \theta/\log 4$, the only difference between the equations \eqref{hhp} and \eqref{he}
is a sign change.
Now, if $n \geq 1$ and if $\nu$ is a probability measure on $U(n)$ such that
$u \mapsto |\log \det(\Id-u)|$ is integrable with respect to $\nu$ (with the same convention
for the logarithm as in Proposition \ref{Ewens}), let us define the {\it capacity}\footnote{
by analogy with the Multiple Input Multiple Output (MIMO) sysems, where the capacity is $\det(\Id+H\tra(H))$,
where $H$ is the rectanglar transmission matrix.} of
$\nu$ as the expectation of $\log \det (\Id -u)$, $u \in U(n)$ following the distribution
$\nu$. A striking fact about the finite-dimensional projections of the Hua-Pickrell distributions is that they maximize the entropy among all the probabilities
which have the same capacity. A similar result has already been proved in \cite{Tsi}
in the context of permutation groups: if the order of the group and the average number of cycles
are fixed, then there exists a unique measure which has the largest entropy, and this measure
is the Ewens measure with a suitable parameter. The equivalent result for unitary matrices is the
following:
\begin{prop} \label{capacity}
Let $n \geq 1$, and let $\delta \in \mathbb{C}$ be such that $\Re(\delta) > -1/2$. Then the capacity
$C_n(\delta)$ of $\mu^{(\delta,n)}$, the projection on $U(n)$ of the Hua-Pickrell measure of parameter $\delta$, is
 well-defined. Moreover, if $f$ denotes the density of $\mu^{(\delta,n)}$ with respect to the
Haar measure, if $\nu$ is a measure on $U(n)$ which is absolutely continuous with respect to the
Haar measure, with density $g$, and if the capacity of $\nu$ is well-defined
and equal to $C_n(\delta)$, then the entropy of $g$ is smaller than or equal to the entropy of
$f$, i.e. the integral of $-g \log g$ with respect to the Haar measure (which is well-defined
in $\mathbb{R} \cup \{- \infty \}$) is smaller than or equal the integral of $-f \log f$, which is
finite. The equality holds only if $f=g$ almost everywhere, i.e. if $\nu$ is equal to $\mu^{(\delta,n)}$.
\end{prop}
\begin{proof}
The integrability condition is equivalent to the fact that, for $u \in U(n)$ following the
Haar measure:
$$\mathbb{E} \left[  |\log \det(\Id-u)| \, \det(\Id-u)^{\bar{\delta}} \det(\Id-\bar{u})^{\delta}
\right] < \infty.$$
By the results of \cite{BNR}, it is equivalent to prove that:
$$
\mathbb{E} \left[ \left|\sum_{k=1}^n \log (1 - \langle e_k, x_k \rangle)
\, \prod_{k=1}^n (1 - \langle e_k, x_k \rangle)^{\bar{\delta}}
(1 - \overline{\langle e_k, x_k} \rangle)^{\delta} \right| \right] <\infty,$$
where the $(x_k)_{1 \leq k \leq n}$ are independent, $x_k$ uniform on the complex
unit sphere of $\mathbb{C}^k$. Then, it is sufficient to have, for all $k \in \{1, \dots ,n\}$:
$$ \mathbb{E} \left[ \left| (1 - \langle e_k, x_k \rangle)^{\bar{\delta}}
(1 - \overline{\langle e_k, x_k \rangle} \rangle)^{\delta} \right| \right] < \infty,
\mathbb{E} \left[ \left|  \log (1 - \langle e_k, x_k \rangle)
 (1 - \langle e_k, x_k \rangle)^{\bar{\delta}}
(1 - \overline{\langle e_k, x_k \rangle})^{\delta} \right| \right] < \infty.$$
 These integrability conditions are implied by:
$$\mathbb{E} \left[ |1 - \langle e_k, x_k \rangle|^{2 \Re(\delta)} \right] < \infty,
\mathbb{E} \left[ \left|\log |1 - \langle e_k, x_k \rangle| \right| \,
|1 - \langle e_k, x_k \rangle|^{2 \Re(\delta)} \right] < \infty,$$
and one checks that these conditions are satisfied for all $\delta \in \mathbb{C}^*$ such that
 $\Re(\delta) > -1/2$. If $f$ and $g$ are defined as in Proposition \ref{capacity}, the
integrability of $f \log f$ under the Haar measure, which implies the finiteness of the entropy
of $\mu^{(n)}_{\delta}$, can be proven in a similar way.

Concerning the optimality to be proven, it results from the elementary inequality
$$ g \log g > f \log f + (g-f)(1+\log f),$$
for all positive $f\neq g$, as shown by a direct study of the function $f\mapsto f \log f + (g-f)(1+\log f)$.
One deduces that under the Haar measure, and for $f$ not almost everywhere equal to $g$,
$$ \mathbb{E} [ - g \log g ] < \mathbb{E} [- f \log f] + \mathbb{E}[(g-f)(1 + \log f)].$$
Moreover, as $\nu$ and $\mu^{(\delta,n)}$ have the same capacity,
the expectation $\mathbb{E}[(g-f)(1+ \log f)]$ is well-defined and equal to zero:
it is also $\E[(g-f)\log f]$ (because $\E(f)=\E(g)=1$ for probability densities), which is exactly a multiple of the difference of the capacities,
thanks to the particular form of $\log f(u)=(\delta+\overline{\delta})\det(\Id-u)$.
\end{proof}

Now, in the next section, we go back to virtual isometries and we prove a strong convergence
result for their eigenangles.
\section{Strong convergence of the eigenangles} \label{conv}

As it was seen above, virtual isometries provide us with the possibility to define on the same probability
space random matrix models for all the finite dimensions. It is then possible to prove strong results, i.e.
properties of almost sure convergence when the dimension goes to infinity. In the particular case of
virtual permutations, Tsilevich \cite{Tsi} proved the following result:
\begin{prop} \label{tsilevich} Let $\sigma = (\sigma_n)_{n \geq 1}$ be a virtual
permutation following the Ewens measure of
parameter $\theta  > 0$, and for $n \geq 1$, $p \geq 1$, let $\ell_p(\sigma_n)$ be
the length of the $p$-th longest cycle of the permutation $\sigma_n \in \mathcal{S}_n$ (for
$p$ larger than the number of cycles of $\sigma_n$, one defines $\ell_p(\sigma_n) := 0$).
Then, almost surely, for all $p \geq 1$, the limit:
$$y_p(\sigma):=\lim_{n\to\infty}\frac{\ell_p(\sigma_n)}{n}$$
exists, and $(y_p(\sigma))_{p \geq 1}$ follows a Poisson-Dirichlet distribution\footnote{For explicit 
formulas and characterizations of Poisson-Dirichlet distributions, see \cite{Pitman}}
of parameter $\theta$.
\end{prop}
\noindent
From Proposition \ref{tsilevich}, it is not difficult to deduce the following result, giving
an almost sure convergence for the eigenangles of the sequence of permutations matrices
 associated to a virtual
permutation.
\begin{prop} \label{convtsilevich}
Let $(u_n)_{n \geq 1}$ be a random virtual isometry consisting of the sequence of permutation
matrices associated to a virtual permutation $\sigma$ which follows the Ewens measure of parameter $\theta > 0$.
Then, for all $n \geq 1$, zero is an eigenangle of $u_n$, and its multiplicity increases almost
surely to infinity when $n$ goes to infinity. Moreover, for $n \geq 1$, $k \geq 1$, let
$\theta^{(n)}_k$ be the $k$-th smallest strictly positive eigenangle of $u_n$, and
$\theta^{(n)}_{-k}$ the $k$-th largest strictly negative eigenangle of $u_n$. Then,
almost surely, for all $n \geq 1$, $k \geq 1$, $\theta^{(n)}_{-k} = - \theta^{(n)}_k$, and for $n$
going to infinity,
$n \theta^{(n)}_k/2 \pi$ converges to the $k$-th smallest element of the set which contains exactly all
the strictly positive multiples of $1/y_p(\sigma)$ for all $p \geq 1$, where $y_p(\sigma)$ is
defined in Proposition \ref{tsilevich}.
\end{prop}
\noindent
In this section, we give a direct and purely probabilistic proof of an analog of Proposition \ref{convtsilevich}
for random virtual isometries which follow  the Haar measure. 
\begin{prop} \label{main}
Let $(u_n)_{n \geq 1}$ be a random virtual isometry, following the Haar
 measure. For $n \geq 1$, $k \geq 1$, let $\theta^{(n)}_k$ be the $k$-th smallest strictly positive
eigenangle of $u_n$, and let $\theta^{(n)}_{1-k}$ be the $k$-th largest nonnegative eigenangle of $u_n$.
Then almost surely, for all $k \in \mathbb{Z}$, $n \theta^{(n)}_k / 2 \pi$ converges
 to a limit $x_k$ when $n$ goes
to infinity, with the following rate:
$$n \theta^{(n)}_k / 2 \pi = x_k + O(n^{-\epsilon}),$$
for some universal constant $\epsilon > 0$. Moreover, the point process $(x_k)_{k \in \mathbb{Z}}$ is a determinantal
process and its kernel $K$ is the {\it sine kernel}, i.e it is given by:
$$K(x,y) = \frac{\sin(\pi (x-y))}{\pi(x-y)}.$$
\end{prop}
\noindent
In the proof of Proposition \ref{main}, the first step is to find an explicit relation between
the characteristic polynomials of $u_n$ and $u_{n+1}$, when $(u_n)_{n \geq 1}$ is a virtual isometry.
\begin{prop} \label{characteristic}
Let $(u_n)_{n \geq 1}$ be a virtual isometry, and for $n \geq 1$, let $x_n := u_n(e_n)$,
$v_n := x_n - e_n$, let $(f_k^{(n)})_{1 \leq k \leq n}$ be an orthonormal basis of $\mathbb{C}^n$,
consisting of eigenvectors of $u_n$, let $(\lambda_k^{(n)})_{1 \leq k \leq n}$ be the
corresponding sequence of
eigenvalues, let $P_n$ be the characteristic polynomial of $u_n$, given by
$$P_n(z) := \det(z \Id_n - u_n),$$
and let us decompose the vector $x_{n+1} \in \mathbb{C}^{n+1}$ as follows:
$$x_{n+1} = \sum_{k=1}^n \mu_{k}^{(n)} f_k^{(n)} + \nu_n e_{n+1},$$
Then for all $n \geq 1$ such that $x_{n+1} \neq e_{n+1}$, one has $\nu_n \neq 1$, and
the polynomials $P_n$ and $P_{n+1}$ satisfy the relation:
$$P_{n+1}(z) = \frac{P_n(z)}{\overline{\nu_n} - 1} \left[ (z - \nu_n)( \overline{\nu_n} - 1)
 - (z-1) \sum_{k=1}^{n} \,   | \mu^{(n)}_k |^2  \frac{ \lambda^{(n)}_k }{ z-\lambda^{(n)}_k}
   \right],$$
for all $z \notin \{\lambda^{(n)}_1, \dots, \lambda^{(n)}_n\}$.
\end{prop}
\begin{proof}
Since $(u_n)_{n \geq 1}$ is a virtual isometry and $x_{n+1} \neq e_{n+1}$, one
has $u_{n+1} = r_{n+1} (u_n \oplus 1),$
where $r_{n+1}$ is the unique reflection such that $r_{n+1}(e_{n+1}) = x_{n+1}$.
One can check that the matrix $r_{n+1}$ is given by:
$$r_{n+1} = \Id_{n+1} + \frac{1}{\overline{\nu_n} - 1} v_{n+1} \overline{v_{n+1}}^t,$$
which implies, for $z \notin \{\lambda^{(n)}_1, \dots \lambda^{(n)}_n, 1\}$,
\begin{align*}
P_{n+1}(z) & = \det(z \Id_{n+1} - u_{n} \oplus 1) \det \left[\Id_{n+1} -
\left( \frac{1}{\overline{\nu_n} - 1}  (z \Id_{n+1} - u_{n} \oplus 1)^{-1}  v_{n+1} \overline{v_{n+1}}^t (u_{n} \oplus 1)
\right) \right] \\
& = (z-1)P_n(z) \left[ 1 - \frac{1}{\overline{\nu_n} - 1} \,
\operatorname{Tr} \left( (z \Id_{n+1} - u_{n} \oplus 1)^{-1} v_{n+1} \overline{v_{n+1}}^t (u_{n} \oplus 1)
\right) \right],
\end{align*}
\noindent
since $\det(\Id + A) = 1 + \operatorname{Tr}(A)$ for any matrix $A$ with rank one. One deduces, by writing
the matrices in the basis $(e_{n+1}, f^{(n)}_1, \dots, f^{(n)}_n)$:
$$P_{n+1}(z)  = (z-1)P_n(z) \left[1 - \frac{1}{\overline{\nu_n} - 1}
\left( \frac{|\nu_n - 1|^2}{z-1} +
\sum_{k=1}^n |\mu_k^{(n)}|^2 \, \frac{\lambda^{(n)}_k}{z - \lambda^{(n)}_k} \right)\right],$$
which implies Proposition \ref{characteristic} for $z \neq 1$. The case $z =1$ can then be deduced from
the fact that $P_{n}$ and $P_{n+1}$ are polynomial functions.
\end{proof}
\noindent
From Proposition \ref{characteristic}, it is possible to deduce some
 information on the behavior of the eigenangles corresponding
to a virtual isometry:
 \begin{prop} \label{Phi}
 Let $(u_n)_{n \geq 1}$ be a virtual isometry, such that, with the notation of
Propositions \ref{main} and \ref{characteristic}, the event
$E_0 := \{ \theta^{(1)}_0 \neq 0\} \cap  \{ \forall n \geq 1, \nu_n \neq 0 \} \cap
\{ \forall n \geq 1, k \in \{1,2,\dots,n\}, \mu^{(n)}_k \neq 0\}$
  holds. Then, for all $n \geq 1$, $k \in \{1, \dots, n\}$,
\begin{align*}
\rho_n &:= |\nu_n| \in (0,1),\\
\psi_n &:= \operatorname{Arg}(\nu_n) \in (-\pi, \pi],\\
\gamma^{(n)}_k &:= \frac{|\mu_k^{(n)}|^2}{1 - |\nu_n|^2}
\end{align*}
are well-defined, and
$
\sum_{k=1}^n \gamma^{(n)}_k = 1.
$
 Moreover, for all $n \geq 1$,
 the expression
   $$ \Phi (\eta) := (1 + \rho_n^2) \, \cos(\eta/2) -  2 \rho_n \cos \left( \eta/2 - \psi_n \right)
 + \, (1 - \rho_n^2) \, \sin(\eta/2) \, \sum_{k=1}^{n} \gamma_k^{(n)} \cot \left( \frac{ \eta
- \theta^{(n)}_k}{2} \right),$$
 which is well-defined for $\eta \in [0,2\pi] \backslash \{\theta^{(n)}_k, 1 \leq k \leq n\}$,
 vanishes if and only if $\eta = \theta^{(n+1)}_k$ for some $k \in \{1,\dots,n+1\}$, and one has the
inequalities:
 \begin{equation}
 0 < \theta^{(n+1)}_1 < \theta^{(n)}_1 < \theta^{(n+1)}_2 < \theta^{(n)}_2 < \dots <
  \theta^{(n+1)}_n < \theta^{(n)}_n < \theta^{(n+1)}_{n+1} < 2 \pi.  \label{interlace}
 \end{equation}
 In other words, the eigenvalues of $u_{n+1}$  interlace between one and the eigenvalues of $u_n$.
 \end{prop}
  \begin{proof}
The quantity $\gamma^{(n)}_k$ is well-defined since $\mu_1^{(n)} \neq 0$, which implies that
$1- |\nu_n|^2 > 0$. The equality $\sum_{k=1}^n \gamma^{(n)}_k = 1$ comes from the fact that $x_{n+1}$ has
norm $1$. By Proposition \ref{characteristic}, one has, for $z \notin \{\lambda^{(n)}_1, \dots,
 \lambda^{(n)}_n \}$,
$$P_{n+1}(z) = \frac{P_n(z)}{\overline{\nu_n} - 1} \left[ (z - \nu_n)( \overline{\nu_n} - 1) - (z-1)
 ( 1 - |\nu_n|^2)
   \sum_{k=1}^{n} \,   \gamma^{(n)}_k   \frac{ \lambda^{(n)}_k }{ z-\lambda^{(n)}_k} \right].$$
Using $\sum_{k=1}^n \gamma^{(n)}_k = 1$ we get
   \begin{align*}
  P_{n+1}(z) & = \frac{P_n(z)}{\overline{\nu_n} - 1}
 \left[ (z - \nu_n)( \overline{\nu_n} - 1) + \frac{  (z-1) ( 1 - |\nu_n|^2)}{2} - \frac{  (z-1) ( 1 - |\nu_n|^2)}{2} \,  \sum_{k=1}^{n} \,   \gamma^{(n)}_k
  \frac{ z+\lambda^{(n)}_k }{ z-\lambda^{(n)}_k} \right]
  \\ & =  \frac{P_n(z)}{\overline{\nu_n} - 1} \left[ z \overline{\nu_n} + \nu_n - \frac{(1+z)
(1 + |\nu_n|^2)}{2}
+   \frac{  (1-z) ( 1 - |\nu_n|^2)}{2} \,  \sum_{k=1}^{n} \,
   \gamma^{(n)}_k   \frac{ z+\lambda^{(n)}_k }{ z-\lambda^{(n)}_k} \right].
   \end{align*}
   \noindent
  Let us now fix arbitrarily a convention for the square root of complex numbers (for example,
for $z \neq 0$, we can define $\sqrt{z}$ as the square root of $z$ which has an argument in the
interval $(-\pi/2,\pi/2]$). We can then define, for all $z \in \mathbb{C}^*$:
  $$Q_n(z) := (\sqrt{z})^{-n} \sqrt{ \overline{P_n (0)}} \, P_n(z).$$
Since
 $$P_{n+1} (0)= P_n(0) \, \frac{\nu_n-1}{\overline{\nu_n} - 1},$$
 one has, for all $z \notin \{0, \lambda^{(n)}_1, \dots, \lambda^{(n)}_n \}$,
 \begin{align*}
 Q_{n+1}(z)  =& \pm  \sqrt{ \frac{ \overline{\nu_n} - 1}{ z (\nu_n - 1)}}
 \frac{Q_n(z)}{  \overline{\nu_n} - 1} \left[ z \overline{\nu_n} + \nu_n - \frac{(1+z)(1 + |\nu_n|^2)}{2}
 \right. \\ & \left. + \,  \frac{  (1-z) ( 1 - |\nu_n|^2)}{2}
 \,  \sum_{k=1}^{n} \,   \gamma^{(n)}_k   \frac{ z+\lambda^{(n)}_k }{ z-\lambda^{(n)}_k} \right]
 \\ =& \pm  \frac{Q_n(z)}{| \nu_n - 1|}  \left[ \sqrt{z} \overline{\nu_n} +
\frac{\nu_n}{\sqrt{z}} - \frac{(\sqrt{z}+\frac{1}{\sqrt{z}})(1 + |\nu_n|^2)}{2}
 \right. \\ & \left. - \,  \frac{  (\sqrt{z} - \frac{1}{\sqrt{z}}) ( 1 - |\nu_n|^2)}{2}
 \,  \sum_{k=1}^{n} \,   \gamma^{(n)}_k   \frac{ \sqrt{z}/ \sqrt{\lambda^{(n)}_k}
 +\sqrt{\lambda^{(n)}_k}/ \sqrt{z}  }{ \sqrt{z}/ \sqrt{\lambda^{(n)}_k}
 - \sqrt{\lambda^{(n)}_k}/ \sqrt{z}  } \right].
   \end{align*}
  \noindent
   Since $\nu_n = \rho_n e^{i \psi_n}$ for $\rho_n \in (0,1)$ and $\psi_n \in (-\pi, \pi]$,
 one has, for all $\eta \in [0,2\pi] \backslash \{\theta^{(n)}_k, 0 \leq k \leq n\}$,
   \begin{align}
   Q_{n+1}(e^{i \eta})  & = \pm
 \frac{Q_n(e^{i \eta})}{| \rho_n e^{i \psi_n} - 1|}
\, \left[ 2 \rho_n \cos \left( \eta/2 - \psi_n \right) - (1 + \rho_n^2) \, \cos(\eta/2)
 \right. \nonumber \\ & \left. - \, (1 - \rho_n^2) \, \sin(\eta/2) \,
 \sum_{k=1}^{n} \gamma_k^{(n)} \cot \left( \frac{ \eta - \theta^{(n)}_k}{2} \right) \right]. \label{Q}
 \end{align}
  \noindent
Now, let us suppose that
  \begin{equation}
  0 < \theta^{(n)}_1 < \theta^{(n)}_2 < \dots < \theta^{(n)}_n < 2\pi \label{strict}
  \end{equation}
   for some $n \geq 1$. The function $\Phi$, given in Proposition \ref{Phi}, is well-defined and
  continuous on each of the
 intervals $[0, \theta^{(n)}_1), (\theta^{(n)}_1, \theta^{(n)}_2),
 \dots (\theta^{(n)}_{n-1}, \theta^{(n)}_n), (\theta^{(n)}_n, 2 \pi]$. Now,
  $\Phi(0)= 1+\rho_n^2 - 2 \rho_n \cos \left( \psi_n \right) \geq (1-\rho_n)^2 > 0$ (note that
$\rho_n < 1$, since $E_0$ holds), for all
  $k \in \{1,\dots, n \}$,
  $\Phi(\eta)$ tends to $-\infty$ when $\eta$ tends to $\theta^{(n)}_k$ from
 below and to $+\infty$ when $\eta$ tends to $\theta^{(n)}_k$ from above
  (since $1- \rho_n^2$, $\sin (\eta/2)$ and $\gamma^{(n)}_k$ are strictly positive), and
  $\Phi(2\pi) = - \Phi(0) < 0$. One deduces that $\Phi(\eta)$
  vanishes at least once on each of the intervals
  $(0, \theta^{(n)}_1), (\theta^{(n)}_1, \theta^{(n)}_2), \dots (\theta^{(n)}_{n-1}, \theta^{(n)}_n), (\theta^{(n)}_n, 2 \pi)$, in other
  words, there exists $(\tau_k)_{1 \leq k \leq n+1}$ such that
  $$0 < \tau_1 < \theta^{(n)}_1 < \tau_2 < \theta^{(n)}_2 < \dots < \tau_n < \theta^{(n)}_n < \tau_{n+1} < 2 \pi $$
  and $\Phi(\tau_k) = 0$ for all $k \in \{1,\dots,n\}$.
 Now, by \eqref{Q}, $Q_{n+1} (e^{i \tau_k}) = 0$, and
 then $P_{n+1} (e^{i \tau_k}) = 0$, for all $k \in \{1,2,\dots,n+1\}$. Hence, necessarily,
 $\tau_k = \theta^{(n+1)}_k$, which implies \eqref{interlace}. In particular,
$$0 < \theta^{(n+1)}_1 < \theta^{(n+1)}_2 < \dots < \theta^{(n+1)}_n  < \theta^{(n+1)}_{n+1} < 2 \pi, $$
and by induction, \eqref{strict} and \eqref{interlace} hold for all $n \geq 1$.
  \end{proof}
\noindent
In the following proposition, we define an event, which is almost surely satisfied under the
Haar measure, and which is involved in a crucial way in our proof of Proposition \ref{main}.
\begin{prop}
Let us suppose that $(u_n)_{n \geq 1}$ is a virtual isometry following the Haar measure, and
let us take the notation above. Then, the event
$E := E_0 \, \cap \, E_1 \, \cap E_2 \, \cap \, E_3$
holds almost surely, where
\begin{align*}
E_1 :=&\{\exists n_0 \geq 1, \, \forall n \geq n_0, \, \rho_n \leq n^{-0.4} \},\\
E_2 :=& \{\exists n_0 \geq 1, \, \forall n \geq n_0, \, \forall
k \in \{1,\dots,n\}, \, \gamma^{(n)}_k \leq n^{-0.99}\},\\
E_3 :=& \{\exists n_0 \geq 1, \, \forall n \geq n_0, \, \forall
k \in \mathbb{Z}, \, n^{-1.7} \leq \theta^{(n)}_{k+1} - \theta^{(n)}_k \leq n^{-0.9}\}.
\end{align*}
\end{prop}
\begin{proof}
It is easy to check that under the Haar measure on virtual isometries:
\begin{itemize}
\item The angle $\theta^{(1)}_0$ is uniform on $(-2 \pi, 0]$;
\item For all $n \geq 1$, $\rho_n$ is the square root of a beta variable of parameters $1$ and $n$;
\item For all $n \geq 1$, $\psi_n$ is uniform on $(-\pi, \pi]$;
\item For all $n \geq 1$, $k \in \{1,2,\dots, n\}$, $\gamma^{(n)}_k = |\xi^{(n)}_k|^2$, where the
 vector $(\xi^{(n)}_k)_{1 \leq k \leq n}$ is
uniform on the complex sphere of dimension $n$.
\end{itemize}
Moreover, the random variables $\theta^{(1)}_0$, $(\rho_n)_{n \geq 1}$, $(\psi_n)_{n \geq 1}$, and
the random vectors $(\xi^{(n)})_{n \geq 1}$ are
independent. It is immediate to check that the condition $E_0$ holds almost surely. Hence, by
 Borel-Cantelli lemma, it is sufficient to check that for $n$ going to infinity:
\begin{align}
&\mathbb{P} [\rho_n > n^{-0.4}] = O(n^{-1.1}), \label{EE1}\\
&\mathbb{P} [\exists k \in \{1,\dots, n\}, \gamma^{(n)}_k > n^{-0.99} ] =  O(n^{-1.1}),\label{EE2}\\
&\mathbb{P} [\exists k \in \mathbb{Z} , \theta^{(n)}_{k+1} - \theta^{(n)}_k > n^{-0.9}] = O(n^{-1.1}),\label{EE3}\\
&\mathbb{P} [\exists k \in \mathbb{Z} , \theta^{(n)}_{k+1} - \theta^{(n)}_k < n^{-1.7}] = O(n^{-1.1}). \label{EE4}
\end{align}

 Now, for all $n \geq 1$, $k \in  \{1,\dots,n+1\}$, $\rho_n^2$ and $\gamma^{(n+1)}_k$ have the same law as
$$ \frac{e_1}{e_1 + e_2 + \dots + e_{n+1}},$$
where $(e_k)_{1 \leq k \leq n+1}$ are independent standard exponential variables. Now, it is a classical
result that
the probabilities  $\mathbb{P} [e_1 \geq n^{0.001}]$ and $\mathbb{P}[e_1 + \dots + e_n \leq n/2]$ decrease
to zero faster than any negative power of $n$ when $n$ goes to infinity. Hence,
$$\mathbb{P} \left[ \frac{e_1}{e_1 + e_2 + \dots + e_{n+1}} > n^{-0.99} \right]  = O(n^{-28}),$$
which easily implies \eqref{EE1} and \eqref{EE2}.
Now, the process of the eigenvalues of a random matrix following the Haar measure
 on $U(n)$ is a determinantal process, with kernel
equal to $K^{(n)}(x)=\sin (nx/2)/ [2 \pi \sin(x/2)]$, where $x$ denotes the difference
 between the two eigenangles which are considered.
One deduces the following estimate for the two-point correlation function
$$
\rho_2^{(n)}(u,v)=K^{(n)}(0)^2-K^{(n)}(u-v)^2=\OO\left(n^4(u-v)^2\right).
$$
The
probability that there exist two eigenangles with distance
smaller than or equal to $x$ is therefore dominated by
$$
\iint_{|u-v|<x}\rho_2^{(n)}(u,v)\dd u\dd v=\OO(n^4x^3),
$$
which implies \eqref{EE4}. In order
 to prove \eqref{EE3}, let us denote by $I$ a measurable
subset of the interval $[0, 2 \pi)$. The probability $\mathbb{P}^{(n)}_I$ that all
 the eigenangles of $u_n$ are
 in $I$ is (by the Andreiev-Heine identity, see e.g. \cite{Sosh})
$$
\mathbb{P}^{(n)}_I
=
\det (M^{n,I}_{j,k})_{1 \leq j,k, \leq n},\
\mbox{where}\
M^{n,I}_{j,k} = \frac{1}{2 \pi} \int_I e^{i(j-k)\theta} \dd \theta.$$
Since for all $\theta \in [0, 2\pi)$, the matrix $(e^{i(j-k) \theta} )_{1 \leq j,k \leq n}$ is
 hermitian and positive (its rank is one and its trace is $n$),
 $M^{n,I}$ and $M^{n,I^c}$ are also hermitian and positive. Moreover,
$M^{n,I} + M^{n, I^c} = \Id_{n}$, hence, the eigenvalues $(\tau_j)_{1 \leq j \leq n}$ of $M^{n, I^c}$ are
in the interval $[0,1]$.
One deduces
$$
\mathbb{P}^{(n)}_I  = \det(M^{n,I})
= \prod_{j=1}^n (1 - \tau_j) \leq \exp \left( - \sum_{j=1}^n \tau_j \right) \leq \exp \left( - \operatorname{Tr} (M^{n, I^c}) \right)
= \exp \left( - n \lambda (I^c)/ 2 \pi \right),
$$
where $\lambda$ is Lebesgue measure. Now, let us choose an integer $q \in [ 13 n^{0.9}, 14 n^{0.9}]$. For
 all $l \in \{0,1, \dots q-1\}$,
$$ \mathbb{P}^{(n)}_{[2 \pi l/q , 2 \pi (l+1)/ q )^c} \leq e^{ -n/q} \leq e^{- n^{0.1}/14}, $$
and then, with probability greater than or equal to $1 - 14 n^{0.9} e^{- n^{0.1}/14}$, $u_n$ has at least
an eigenangle in each interval of the
form $[2 \pi l/q , 2 \pi (l+1)/ q )$.
In this case, the maximal distance between two eigenangles is smaller than or equal to
$4 \pi / q \leq 4 \pi n^{-0.9} /13$, which implies \eqref{EE3}.
\end{proof}

Our interest in the event $E$ lies in the following result:
\begin{prop} \label{convergence}
Let $(u_n)_{n \geq 1}$ be a virtual isometry such that the event $E$ holds. Let us
extend the notation $\gamma^{(n)}_k$ to all the values $k \in \mathbb{Z}$, in the unique way such that
$\gamma^{(n)}_{k+n} = \gamma^{(n)}_k$.  Then, for all $k \in \mathbb{Z}$, there exists 
 $L \neq 0$ such that
$$ \theta^{(n)}_k \, \exp \left( \sum_{p=1}^{n-1}  \gamma^{(p)}_k \right) = L + O(n^{-\epsilon})$$
when $n$ goes to infinity, $\epsilon > 0$ being a universal constant. 
\end{prop}
\begin{proof}
Let us suppose $n \geq |k|+1$. Then,
$\theta^{(n+1)}_k \in (\theta^{(n)}_{k-1}, \theta^{(n)}_k)$ for $k \geq 2$, $\theta^{(n+1)}_k \in (0, \theta^{(n)}_1)$ for
$k = 1$, $\theta^{(n+1)}_k \in (\theta^{(n)}_0, 0) = (\theta^{(n)}_n - 2 \pi, 0)$ for $k=0$,
and $\theta^{(n+1)}_k \in (\theta^{(n)}_{k}, \theta^{(n)}_{k+1}) = (\theta^{(n)}_{n+k} - 2 \pi,
\theta^{(n)}_{n+k+1} - 2 \pi) $ for $k \leq -1$. Moreover,  by Proposition \ref{Phi},
 $$(1 + \rho_n^2) \, \cos(\theta^{(n+1)}_k/2) -  2 \rho_n \cos \left(  (\theta^{(n+1)}_k/2) - \psi_n \right) $$
 \begin{equation}
  + \, (1 - \rho_n^2) \, \sin(\theta^{(n+1)}_k/2) \, \sum_{j=1}^{n} \gamma_j^{(n)} \cot
  \left( \frac{ \theta^{(n+1)}_k - \theta^{(n)}_j}{2} \right) = 0. \label{equ}
 \end{equation}
for  $k \geq 1$, and
$$(1 + \rho_n^2) \, \cos(\theta^{(n+1)}_{n+k+1}/2) -  2 \rho_n \cos \left(  (\theta^{(n+1)}_{n+k+1}/2) - \psi_n \right) $$
 \begin{equation}
  + \, (1 - \rho_n^2) \, \sin(\theta^{(n+1)}_{n+k+1}/2) \, \sum_{j=1}^{n} \gamma_j^{(n)}
  \cot \left( \frac{ \theta^{(n+1)}_{n+k+1} - \theta^{(n)}_j}{2} \right) = 0.\nonumber
 \end{equation}
 for $k \leq 0$, which also implies \eqref{equ}, since $\sin(x+\pi) = -\sin(x)$ and $\cos(x+\pi) = -\cos(x)$ for all $x \in \mathbb{R}$.
Then, by the periodicity of the cotangent, one has:
 $$(1 + \rho_n^2) \, \cos(\theta^{(n+1)}_k/2) -  2 \rho_n \cos \left(  (\theta^{(n+1)}_k/2) - \psi_n \right) $$
 \begin{equation}
  + \, (1 - \rho_n^2) \, \sin(\theta^{(n+1)}_k/2) \, \sum_{j \in J} \gamma_j^{(n)}
  \cot \left( \frac{ \theta^{(n+1)}_k - \theta^{(n)}_j}{2} \right) = 0.  \end{equation}
  for any set $J$ consisting of $n$ consecutive integers.
   From now, we choose $J$ equal to the set of integers $j$
  such that $\theta^{(n)}_j \in (\theta^{(n+1)}_k - \pi, \theta^{(n+1)}_k + \pi]$.
Since the condition $E$ holds, one has the following estimates (for $k$ fixed and $n$ going to infinity):
$$1 + \rho_n^2 = 1 + O(n^{-0.8}),$$
$$ \left| 2 \rho_n \cos \left(  (\theta^{(n+1)}_k/2) - \psi_n \right) \right| \leq 2 \rho_n = O(n^{-0.4}),$$
$$ \theta^{(n+1)}_k =  O(n^{-0.9}),$$
$$ \cos(\theta^{(n+1)}_k/2) = 1 + O(n^{-1.8}),$$
$$ (1 + \rho_n^2) \, \cos(\theta^{(n+1)}_{k}/2) -  2 \rho_n \cos \left(  (\theta^{(n+1)}_{k}/2) -
\psi_n \right) = 1 + O(n^{-0.4}),$$
 and then,
 $$ (1 - \rho_n^2) \, \sin(\theta^{(n+1)}_{k}/2) \, \sum_{j \in J} \gamma_j^{(n)} \cot \left(
 \frac{ \theta^{(n+1)}_{k} - \theta^{(n)}_j}{2} \right) = - 1 + O(n^{-0.4})$$
 Since $1-\rho_n^2 = 1 + O(n^{-0.8})$, and $ \sin(\theta^{(n+1)}_{k}/2) =
 \left(\theta^{(n+1)}_{k}/2 \right) \left( 1 + O(n^{-1.8}) \right)$,
 one deduces:
 $$ \theta^{(n+1)}_{k} \, \sum_{j \in J} \gamma_j^{(n)} \cot \left( \frac{ \theta^{(n+1)}_{k} -
 \theta^{(n)}_j}{2} \right) = - 2 + O(n^{-0.4})$$
 Now, since the function $x \rightarrow \cot(x) - 1/x$ is bounded on the interval $[-\pi/2, \pi/2]$,
 $$ \theta^{(n+1)}_{k} \, \sum_{j \in J} \gamma_j^{(n)} \left[ \cot \left( \frac{ \theta^{(n+1)}_{k} -
 \theta^{(n)}_j}{2} \right) - \frac{2}{ \theta^{(n+1)}_{k} - \theta^{(n)}_j} \right]$$
 is dominated by
 $$|\theta^{(n+1)}_{k}| \, \sum_{j \in J} \gamma_j^{(n)} = |\theta^{(n+1)}_k| = O(n^{-0.9}),$$
 which implies:
  $$ \theta^{(n+1)}_{k} \, \sum_{j \in J}  \frac{\gamma_j^{(n)} }{ \theta^{(n+1)}_{k} - \theta^{(n)}_j} = - 1 + O(n^{-0.4}).$$
 Note that $$\theta^{(n)}_{k+n+1} = \theta^{(n)}_{k+1} + 2 \pi \geq \theta^{(n+1)}_k + 2 \pi$$
 and
 $$\theta^{(n)}_{k-n-1} = \theta^{(n)}_{k-1} - 2 \pi \leq \theta^{(n+1)}_k - 2 \pi,$$
 hence, all the elements of $J$ are included in the interval $[k-n, k+n]$. Moreover,  for $n$ large enough and for all integers $p \geq 1$:
 $$\theta^{(n)}_{k + p + 1} - \theta^{(n+1)}_k \geq \theta^{(n)}_{k + p + 1} - \theta^{(n)}_{k+1} \geq p  n^{-1.7}$$
 and
 $$\theta^{(n+1)}_{k} - \theta^{(n)}_{k - p - 1} \geq \theta^{(n)}_{k - 1} - \theta^{(n)}_{k - p - 1} \geq p  n^{-1.7}.$$
 One deduces that, for $n$ large enough:
 \begin{align*}
  \left | \theta^{(n+1)}_{k} \, \sum_{j \in J, |j - k| > 1}  \frac{\gamma_j^{(n)} }{ \theta^{(n+1)}_{k} - \theta^{(n)}_j} \right|
 & \leq 2 \,|\theta^{(n+1)}_{k}| \, \left( \sup_{1 \leq j \leq n}  \gamma_j^{(n)} \right) \, \left(\sum_{p=1}^n \frac{1}{p n^{-1.7}} \right)
 \\ & = O (n^{-0.9}) \, . \, O(n^{-0.99}) \, . \, O (n^{1.7} \log n) = O(n^{-0.1}).
 \end{align*}
 On the other hand
 $$ 0 \leq \theta^{(n+1)}_{k} - \theta^{(n)}_{k - 1} \leq \theta^{(n)}_{k+1} - \theta^{(n)}_{k-1} = O(n^{-0.9})$$
 and
 $$ 0 \leq \theta^{(n)}_{k+1} - \theta^{(n+1)}_{k} \leq \theta^{(n)}_{k+1} - \theta^{(n)}_{k-1} = O(n^{-0.9}),$$
hence, $k-1, k, k+1 \in J$ for $n$ large enough. One deduces:
 $$\theta_k^{(n+1)} \left( \frac{\gamma^{(n)}_{k-1}}{ \theta^{(n+1)}_k - \theta^{(n)}_{k-1}} +  \frac{\gamma^{(n)}_{k}}{ \theta^{(n+1)}_k - \theta^{(n)}_{k}}
 +  \frac{\gamma^{(n)}_{k+1}}{ \theta^{(n+1)}_k - \theta^{(n)}_{k+1}} \right) = -1 + O(n^{-0.1}).$$
 Therefore, for $n$ large enough, there exists $j \in \{k-1,k, k+1\}$ such that
 $$\frac{\theta^{(n+1)}_k \gamma^{(n)}_{j}}{  \theta^{(n+1)}_k - \theta^{(n)}_{j}}  < -1/4.$$
 One deduces that $\theta^{(n+1)}_k - \theta^{(n)}_{j} = O(n^{-1.89})$, since $\theta^{(n+1)}_k
 \gamma^{(n)}_{j} =  O(n^{-1.89})$.
Now, let us suppose that $j = k+1$. In this case, $\theta^{(n+1)}_k - \theta^{(n)}_j$ is negative,
and then $\theta^{(n+1)}_k$ should be positive, and $k$ should
be strictly positive. Then, $\theta^{(n)}_k - \theta^{(n+1)}_k > 0$ and
 $$\theta^{(n)}_j - \theta^{(n+1)}_{k} \geq (\theta^{(n)}_{k+1} - \theta^{(n)}_{k} ) +
 (\theta^{(n)}_k - \theta^{(n+1)}_k)
 \geq n^{-1.7},$$
 which is a contradiction for $n$ large enough. Similarly, if  $j = k-1$,  $\theta^{(n+1)}_k - \theta^{(n)}_j$
 is positive, $\theta^{(n+1)}_k$ is negative,
 $k \leq 0$,  $\theta^{(n)}_k - \theta^{(n+1)}_k < 0$ and $\theta^{(n+1)}_k - \theta^{(n)}_j \geq n^{-1.7}$,
 which is again a contradiction.
 Therefore, $j = k$ for $n$ large enough and $\theta^{(n+1)}_k - \theta^{(n)}_k = O(n^{-1.89})$.
 One deduces that for $j \in \{k-1,k+1\}$,
 $$|\theta^{(n+1)}_k - \theta^{(n)}_{j} | \geq |\theta^{(n)}_k - \theta^{(n)}_{j} | -
 |\theta^{(n+1)}_k - \theta^{(n)}_{k} | \geq n^{-1.7} - O(n^{-1.89})
 \geq n^{-1.7}/2$$
 if $n$ is large enough, and then
 $$\frac{\theta^{(n+1)}_k \gamma^{(n)}_{j}}{  \theta^{(n+1)}_k - \theta^{(n)}_{j}}  = O(n^{-0.1}).$$
Consequently, one has:
$$\frac{\theta^{(n+1)}_k \gamma^{(n)}_{k}}{  \theta^{(n+1)}_k - \theta^{(n)}_{k}}  = -1 + O(n^{-0.1}),$$
and since $\theta^{(n+1)}_k \neq 0$,
$$\frac{\gamma^{(n)}_{k}}{1 - (\theta^{(n)}_k / \theta^{(n+1)}_k )} = -1 + O(n^{-0.1}),$$
which implies
$$\gamma^{(n)}_k = \left( \frac{\theta^{(n)}_k}{\theta^{(n+1)}_k} - 1\right) \left(1 + O(n^{-0.1}) \right).$$
In particular, $(\theta^{(n)}_k/\theta^{(n+1)}_k) - 1$ is equivalent to $\gamma^{(n)}_k$, and then,
dominated by
$n^{-0.99}$. One deduces:
$$\gamma^{(n)}_k = \left( \frac{\theta^{(n)}_k}{\theta^{(n+1)}_k} - 1\right) + O(n^{-0.99}) \,. O(n^{-0.1}),$$
$$ \frac{\theta^{(n)}_k}{\theta^{(n+1)}_k}  = 1 + \gamma^{(n)}_k + O(n^{-1.09}),$$
and then
$$\log \left(\frac{\theta^{(n)}_k}{\theta^{(n+1)}_k} \right) = \gamma^{(n)}_k + O(n^{-1.09}).$$
Now, if one sets
$$L_n := \log (|\theta^{(n)}_k|) + \sum_{p=1}^{n-1} \gamma^{(p)}_k,$$
then $L_{n+1} - L_n = O(n^{-1.09})$. One deduces that $(L_n)_{n \geq 1}$ converges to a limit 
$L_{\infty}$ when $n$ goes to infinity, with $L_{\infty} - L_n = O(n^{-0.09})$. Taking the exponential
ends the proof of Proposition \ref{convergence}, for $\epsilon = 0.09$. 
\end{proof}
\noindent
We have now all the ingredients involved in the proof of Proposition \ref{main}. Indeed,
by Proposition \ref{convergence}, there exists a random variable $L \neq 0$ such that almost surely, for 
$n$ going to infinity, 
\begin{equation} \theta^{(n)}_k \, \exp \left( \sum_{p=1}^{n-1}  \gamma^p_k \right) =  n \theta^{(n)}_k
\, \exp \left( - \log n + \sum_{p=1}^{n-1}  \frac{1}{p} \right)   \exp \left( M_n \right)  =  L + O(n^{-\epsilon}), \label{lll}
\end{equation}
where
$$M_n := \sum_{p=1}^{n-1} \left( \gamma^{(p)}_k - \frac{1}{p} \right),$$
since the event $E$ holds almost surely. Now, the variables $(\gamma^{(n)}_k)_{n \geq 1}$ are
 independent, with expectation $1/n$, and then the process $(M_n)_{n \geq 1}$ is a martingale with
respect to the filtration generated by  $(\gamma^{(n)}_k)_{n \geq 1}$. Moreover, with the notation above,
\begin{align*}
\mathbb{E} \left[ \left( \gamma^{(n)}_k - \frac{1}{n} \right)^2 \right] &
= \mathbb{E} \left[ \left( \frac{e_1}{e_1 + e_2 + \dots e_n} - \frac{1}{n} \right)^2 \right]  \\
& =  \mathbb{E} \left[ \left( \frac{(n-1)e_1 - e_2 - e_3 - \dots - e_n}{n(e_1 + e_2 + \dots e_n)}  \right)^2 \right] \\
& \leq \mathbb{P} [ e_1 + \dots + e_n \leq n/2 ]+
\frac{4}{n^4} \, \mathbb{E} \left[ \left((n-1)e_1 - e_2 - e_3 - \dots - e_n \right)^2 \right]  \\ &
\leq  O(n^{-28})+ \frac{4}{n^4} \, \left[ \operatorname{Var} ((n-1)e_1) +
\operatorname{Var} (e_2) + \dots + \operatorname{Var} (e_n) \right]\\
& =  O(n^{-28}) + \frac{4n(n-1)}{n^4} = O(1/n^2).
\end{align*}
Hence, the martingale $(M_n)_{n \geq 1}$ is bounded in $L^2$, and then, converges
 almost surely (and in $L^2$) to a limit
random variable $M_{\infty}$. More precisely, for $n \geq 1$, 
$$\mathbb{E} [(M_{\infty} - M_n)^2] = \sum_{m = n}^{\infty} 
\mathbb{E} \left[ \left( \gamma^{(m)}_k - \frac{1}{m} \right)^2 \right] 
= O(1/n),$$
and by applying Doob's inequality to the martingale $(M_{2^q + m})_{m \geq 0}$, for $q \geq 0$,
\begin{align*}
\mathbb{E} \left[ \sup_{n \geq 2^q} (M_{\infty} - M_{n})^2 \right] & 
\leq 2 \left( \mathbb{E} [ (M_{\infty} - M_{2^q})^2 ]+ \mathbb{E}\left[ \sup_{n \geq 2^q} (M_{n} - M_{2^q})^2 \right] \right) \\ &
\leq 10 \, \mathbb{E} [ (M_{\infty} - M_{2^q})^2 ] = O(2^{-q}),
\end{align*}
when $q$ goes to infinity.
Hence, 
$$\mathbb{P} \left[ \sup_{n \geq 2^q} |M_{\infty} - M_{n}|  \geq 2^{-q/4} \right] 
\leq 2^{q/2} \, \mathbb{E} \left[ \sup_{n \geq 2^q} (M_{\infty} - M_{n})^2 \right] = O(2^{-q/2}),$$
and by Borel-Cantelli's lemma, 
$$\mathbb{P} \left[ \exists q_0 \geq 1, \, \forall q \geq q_0, \, \forall n \geq 2^{q}, \, |M_{\infty} - M_{n}|  \geq 2^{-q/4} \right] = 1.$$
Therefore, almost surely, 
$$ |M_{\infty} - M_n| = O(n^{-1/4})$$
when $n$ goes to infinity. Now, by \eqref{lll}, one has almost surely:
\begin{align*}
n \theta^{(n)}_k / 2 \pi &  = \frac{1}{2\pi} ( L + O(n^{-\epsilon}) )\exp \left( - M_n \right) 
 \exp \left( \log n - \sum_{p=1}^{n-1}  \frac{1}{p} \right) \\ & = 
\frac{1}{2\pi}  ( L + O(n^{-\epsilon}) ) \exp \left( - M_{\infty} + O(n^{-1/4}) \right) 
  \, \exp  \left( -\gamma + O(n^{-1}) \right) \\ & = x_k + O(n^{- \epsilon}),
  \end{align*}
where  $\gamma$ is Euler constant and $$x_k=
 \frac{L}{2 \pi \, e^{\gamma + M_{\infty} }},$$
if one assumes $\epsilon \leq 1/4$.
Moreover, since one knows
that the point process $(n \theta^{(n)}_k/2 \pi)_{k \in \mathbb{Z}}$ converges weakly to
a determinantal process with sine kernel, the limit point process $(x_k)_{k \in \mathbb{Z}}$ is
necessarily also a determinantal process with sine kernel.
\begin{rem}
This probabilistic proof of almost sure convergence to a sine point process can be extended to Hua-Pickrell measures. This requires showing analogues of equations \eqref{EE1}
 till \eqref{EE4}, hence a precise analysis on the hypergeometric kernel, which is not the purpose of this article.
Proposition \ref{main} is the exact equivalent to Proposition \ref{convtsilevich}, for unitary matrices. The
 link between Proposition \ref{tsilevich} and Proposition \ref{convtsilevich} shows immediately that the
behavior of the large cycles of random permutations is strongly related to the bahaviour of the
 corresponding eigenvalues which are close to $1$.
Similarly, the behavior of the small
cycles of a permutation is directly related to the traces of the small powers of the corresponding
matrix. Now, if for $n, p \geq 1$, $Y^{(n)}_p$ denotes the number of $p$-cycles of a random
permutation on $\mathcal{S}_n$ which follows the Ewens measure of parameter $\theta> 0$, then
for all $p_0 \geq 1$, the joint distribution of $(Y^{(n)}_p)_{1 \leq p \leq p_0}$ tends
to the distribution of $(Y_p)_{1 \leq  p \leq p_0}$, where $(Y_p)_{p \geq 1}$ is a sequence
of independent Poisson random variables such that $\mathbb{E}[Y_p] = \theta/p$.
This result can be easily translated to a result of weak convergence for the finite-dimensional
marginales of the sequence $(\operatorname{Tr} (u^p_n))_{p \geq 1}$, where $u_n$ is a random
permutation matrix of order $n$, which follows the Ewens measure of parameter $\theta$.
For general unitary matrices, one has a similar result on the traces: if
$u_n$ is a random unitary matrix of order $n$ which follows the Haar measure, then
the finite-dimensional marginales of $(\operatorname{Tr} (u^p_n))_{p \geq 1}$ converge
in law to sequences of i.i.d. complex gaussian random variables \cite{DiacShah}. 
\end{rem}

\renewcommand{\refname}{References}


\begin{thebibliography}{99}
\bibitem{BO} A. Borodin, G. Olshanski, Infinite Random Matrices and Ergodic Measures, Comm. Math. Phys., 203 (2001), 87-123.
\bibitem{bhny} P. Bourgade, C.-P. Hughes, A. Nikeghbali, M. Yor, The characteristic polynomial
of a random unitary matrix: a probabilistic approach.  Duke Math. J.,  145  (2008),  no. 1, 45-69.
\bibitem{BNR} P. Bourgade, A. Nikeghbali, A. Rouault,
Hua-Pickrell measures on general compact groups, to appear in S\'eminaire de probabilit\'es.
\bibitem{DiacShah} P. Diaconis, M. Shahshahani, On the eigenvalues
of random matrices, Studies in applied probability,  J. Appl.
Probab.,  31A  (1994), 49-62.
\bibitem{Kerov} S.-V. Kerov, G.-I. Olshanski, A.-M. Vershik, Harmonic
  analysis on the infinite symmetric group, Comptes Rendus de
  l'Acad\'emie des sciences de Paris, 316 (1993), 773-778.
   \bibitem{King75} J.-F.-C. Kingman, Random discrete distribution, 
   J. Roy. Stat. Soc. B, 37 (1975), 1-22.
\bibitem{King78a} J.-F.-C. Kingman, Random partitions in population genetics,
Proc. R. Soc. Lond. (A), 361 (1978), 1-20. 
\bibitem{King78b} J.-F.-C. Kingman, The representation of partition structures, 
J. London Math. Soc. (2), 18 (1978), 374-380. 
\bibitem{Ner} Y.-A. Neretin, Hua type integrals over unitary groups 
and over projective limits of unitary groups, Duke Math. J., 114 (2002), 239-266. 
\bibitem{OV} G. Olshanski, A. Vershik, Ergodic unitarily invariant
 measures on the space of infinite Hermitian matrices, 
 Amer. Math. Soc. Trans., 175 (1996), 137-175.
\bibitem{Pitman}
J. Pitman, Combinatorial Stochastic Processes, Lecture
Notes in Mathematics Vol 1875, Springer-Verlag, Berlin,
2006.
\bibitem{Sosh} A Soshnikov, Determinantal random point fields, Russ. Math. Surv., 55 (2000), no. 5, 923-975.
\bibitem{Tsi} N.-V. Tsilevich,
Distribution of cycle lengths of infinite permutations,
J. Math. Sci. (N.Y.), 87 (1997), no. 6, 4072-4081
\bibitem{Tsi98} N.-V. Tsilevich,
 Stationary Measures on the Space of Virtual Permutations for an Action of the Infinite Symmetric Group (1998)

\end{thebibliography}
\end{document}